%% file: arxiv.tex
\documentclass[english,12pt]{article}

\input{preamble.tex}
\usepackage{geometry}
\tcolorboxenvironment{exmp}{breakable, 
enhanced, 
frame hidden,
capture=minipage,
colback=gray!20!white}
\tcolorboxenvironment{thm}{breakable, 
enhanced, 
frame hidden,
capture=minipage,
colback=Dandelion!40!yellow!20!white}
\tcolorboxenvironment{prop}{breakable, 
enhanced, 
frame hidden,
capture=minipage,
colback=Dandelion!40!yellow!20!white}

\allowdisplaybreaks

\title{Instance-dependent uniform tail bounds for\\ empirical processes}
\usepackage{authblk}

\author{Sohail Bahmani}

\begin{document}
\maketitle
\begin{abstract}
    We formulate a uniform tail bound for empirical processes indexed by a class of functions, in terms of the individual deviations of the functions rather than the worst-case deviation in the considered class. The tail bound is established by introducing an initial ``deflation'' step to the standard generic chaining argument. The resulting tail bound is the sum of the complexity of the ``deflated function class'' in terms of a generalization of Talagrand's $\gamma$ functional, and the deviation of the function instance, both of which are formulated based on the natural seminorm induced by the corresponding Cram\'{e}r functions. Leveraging another less demanding natural seminorm, we also show similar bounds, though with implicit dependence on the sample size, in the more general case where finite exponential moments cannot be assumed. We also provide approximations of the tail bounds in terms of the more prevalent Orlicz norms or their ``incomplete'' versions under suitable moment conditions.
\end{abstract}

\input{main-arxiv.tex}
\input{appendix-arxiv.tex}

\printbibliography

\end{document}

%% file: main-arxiv.tex

\section{Introduction}\label{sec:introduction}
Let $(X_i)_{i=1}^n$ be i.i.d. copies of a random variable $X$ taking values in some space $\mc X$, and denote by $\E_n$ the expectation with respect to the empirical measure associated with the samples $(X_i)_{i=1}^n$. A central question of the theory of empirical processes is to find tail bounds for the empirical average $\E_n f(X)= n^{-1}\sum_{i=1}^n f(X_i)$ that hold uniformly for all functions $f$ belonging to a given function class $\mc F\subset {\mbb R}^{\mc X}$.

Assuming that the functions in $\mc F$ are all zero-mean, the existing tail bounds in the literature typically assert that with probability at least $1-e^{-r}$, for all $f\in \mc F$ we have
\begin{align*}
	\E_n f & \le (\msf C(\mc F) + \msf S_r(\mc F)) o_n(1)\,,
\end{align*}
where $\msf C(\mc F)$ depends on some measure of ``complexity'' of the function class $\mc F$ (e.g., VC-dimension \citep{VC71,Vap98}, Rademacher complexity \citep{vdVW12,KP00}, or Talagrand's functional \citep{Tal14}), and $\msf S_r(\mc F)$ is some notion of the ``worst-case deviation'' of the functions $f\in \mc F$ at the confidence level $e^{-r}$. Our goal in this paper is to establish ``instance-dependent'' tail bounds in which the worst-case deviation above is replaced by the deviation of each particular function of the function class. It turns out that the instance-dependent tail bounds may provide some improvements in terms of the complexity term as well.

A closely related set of results are tail bounds for \emph{ratio type} empirical processes. \citet{GKW03} and \citet{GK06} have developed such tail bounds for processes indexed by a class of $[0,1]$--bounded functions. In particular, various elaborate non-asymptotic tail bounds are derived in \citep{GK06} by ``slicing'' (or ``peeling'') the function class to sets of functions for which the variance proxy is nearly the same, and applying to each slice Talagrand's concentration inequality for uniformly bounded empirical processes.

Our inspiration is a recent result of \citet{LM23} on instance-dependent tail bounds for certain Gaussian processes. \citet{LM23} used this result as a benchmark to motivate their main goal which is robust mean estimation with optimal direction-dependent sub-Gaussian confidence intervals. Specifically, in the case of a Gaussian processes indexed by points in some centered Euclidean ball,  \cite[Proposition 1]{LM23} derived refined tail bounds that depend on the standard deviation at any queried direction rather than the worst-case standard deviation (i.e., largest eigenvalue of the corresponding covariance matrix). Furthermore, the complexity of the entire class is replaced by a quantity which, depending on the confidence level and the spectrum of the covariance matrix, can be significantly smaller than the square root of the trace of the covariance matrix appearing in the standard bounds.

\ref{sec:problem} provides a more precise statement of the problem of interest. The instance-dependent tail bounds under the assumption of finite exponential moments are presented in \ref{sec:main-exponential}. As a complement to this section, our calculations in \ref{apx-sec:Orlicz} to derive more explicit expressions in the more commonly used case of function classes in (exponential type) Orlicz spaces, can be of independent interest. In \ref{sec:Examples} we consider three illustrative examples. In \ref{ssec:Gaussian-marginals} we discuss the problem studied by \citep{LM23} in more details, and in \ref{ssec:CI-Gaussian-top-m} we use instance-dependent bounds to formulate confidence intervals for the $m$-th largest mean of a general Gaussian vector. \ref{sec:main-only-L1} further generalizes the results of \ref{sec:main-exponential} to situations where the functions of interest are $L_1$ (with respect to the law of $X$), and particularly may not have finite exponential moments. As a corollary, these rather general bounds are made more explicit, especially in terms of the sample size, for functions with finite moments of every order which, again, do not necessarily have finite exponential moments.

\section{Preliminaries and Problem Setup}\label{sec:problem}
Let $\mc F$ denote a finite but arbitrarily large subset\footnote{In many situations infinite function classes can be considered as well, but a completely rigorous analysis for the problems of interest requires the measurability issues to be addressed, e.g., as in \cite[Appx. C]{Pol84}.} of a vector space $\mbb V$ of \emph{centered} functions from $\mc X$ to $\mbb R$ whose \emph{cumulant generating function} is finite in a neighborhood of the origin. Specifically, for every $f\in \mbb V$ we have
\begin{align*}
	\E f(X) & = 0\,,
\end{align*}
and $I_f = \{\lambda \in \mbb R\st \log\E e^{\lambda f(X)}< +\infty\}$, the domain of the corresponding cumulant generating function, contains $0$ in its interior.

For simplicity we assume that the zero function, denoted by $0$, is also in $\mc F$. We also frequently use functions $T_r\st\,\mbb V \to \mbb R_{\ge 0}$ that are defined for $r\ge 0$ as
\begin{align*}
	T_r(g) & \defeq \inf_{\lambda\ge 0} \frac{r+\log \E e^{\lambda g(X)}}{\lambda}\,,
\end{align*}
with the convention that for $r>0$, if $\log \E e^{\lambda g(X)}=+\infty$, then the objective of the infimum is also infinite and the corresponding $\lambda$ is implicitly excluded.
These functions determine certain confidence intervals of interest and in fact are inverses of the \emph{rate function}, a central object in the theory of large deviations \citep{Var84, DZ10}, associated with the random variable $g(X)$.
We emphasize that the domain of $T_r(\cdot)$ is not restricted to $\mc F$, and as will be seen in the sequel we also apply $T_r(\cdot)$ to other functions in $\mbb V$.

It is worth mentioning that $T_r(g)$ is a general substitute for many prevalent measures of ``deviation'' for a function $g\in \mbb V$ at the confidence level $e^{-r}$. For example, if $g(X)$ has a sub-Gaussian distribution and the corresponding sub-Gaussian parameter is proportional to $\norm{g}_{L_2}$, then we have $T_r(g)\lesssim \sqrt{r}\norm{g}_{L_2}$.\footnote{Here and throughout, $P\lesssim Q$ is used as a shorthand for the inequality $P\le c Q$ for some absolute constant $c>0$.} Another common example is of bounded functions $g(\cdot)$, where using Bernstein-type bounds (see \ref{lem:Bernstein} in \ref{apx-sec:proofs}) we can show that $T_r(g) \lesssim \sqrt{r}\norm{g}_{L_2} + r\norm{g}_{L_\infty}$. More generally, as detailed in \ref{apx-sec:Orlicz}, for exponential-type Orlicz spaces, $T_r(g)$ can be bounded by the corresponding Orlicz norm.

The function $T_r(\cdot)$ has certain properties that are important in our derivations. We have collected these properties in the following lemma, which is proved in \ref{apx-sec:proofs} to be self-contained. It is worth mentioning that more general alternatives to $T_r(\cdot)$ with similar properties can be defined easily using certain variational approximations of the corresponding quantile functions \cite[Theorem 2.4]{Pin14}. These variational approximations are important in concentration inequalities for sums of independent random variables (see, e.g., \citep{Rio17} and \citep{Mar21}). We use the mentioned less demanding alternatives of $T_r(\cdot)$ in \ref{sec:main-only-L1} to state a more general, but less explicit, version of our results in \ref{sec:main-exponential}.

\begin{lem}[Properties of $T_r(\cdot)$]\label{lem:T-r(g)}	The function $T_r(\cdot)$ has the following properties:
	\begin{enumerate}[label={(\roman*)}]
		\item \label{itm:positive_homogenous} $T_r(\cdot)$ is positive homogenous in the sense that $T_r(\alpha g) = \alpha T_r(g)$ for any $\alpha > 0$ and all functions $g\in \mbb V$.
		\item \label{itm:zero_is_root} $T_0(g) = 0$ for all functions $g$.
		\item \label{itm:concave_subadditive} The mapping $r\mapsto T_r(g)$, for $r>0$ and any particular function $g \in \mbb V$, is concave and subadditive.
		\item \label{itm:envelope} The even envelope of $T_r(\cdot)$ defined as
		      \begin{align}
			      \bar{T}_r(g) & = \max\{T_r(g),T_r(-g)\}\,,\label{eq:T-r-envelope}		      \end{align}
		      is a seminorm.
	\end{enumerate}
\end{lem}

For any fixed $f\in \mbb V$ the following elementary lemma, which is essentially the well-known Chernoff bound, expresses a tail bound for $\mbb E_n f$ in terms of $T_{r/n}(f)$. The proof is provided in the Appendix for completeness.
\begin{lem}\label{lem:Chernoff}
	With the definitions above, for any function $f\in \mbb V$ (whose moment generating function has $0$ in the interior of its domain), with probability at least $1-e^{-r}$, we have
	\begin{align*}
		\E_n f(X) & \le T_{r/n}\left(f\right)\,.
	\end{align*}
\end{lem}

It is natural to seek an extension of \ref{lem:Chernoff} that provides an upper tail bound for the random variable of the form
\begin{align*}
	Z & =\sup_{f\in \mc F}\ \left(\E_n f(X) - T_{r/n}\left(f\right)\right)\,,
\end{align*}
which translates to a uniform bound for $\E_n f$ that holds for every instance of $f\in \mc F$. It is often more convenient to work with tail bounds expressed in terms of some seminorm of $f$ rather than the $T_r(f)$ which is not subadditive. A natural choice is $\bar{T}_r(\cdot)$ defined by \eqref{eq:T-r-envelope}, and consider the seminormed spaces $(\mbb V, \bar{T}_r(\cdot))$ for $r\ge 0$, where the functions of $\mc F$ belong to. In this paper we focus on finding an upper tail bound for random variables of the form
\begin{align}
	Z & = \sup_{f\in \mc F}\ \left(\E_n f(X) - \bar{T}_{r/n}(f)\right)\,.\label{eq:instance-dependent-formulation}
\end{align}
We emphasize that we use the term ``instance-dependent tail bounds'' specifically to refer to the bounds that generalize the Chernoff bound for an individual function, to the entire class as described above. For example, the result of the standard generic chaining arguments can be expressed in a way that the tail bounds depend on the queried function $f$. However, the resulting bounds are in terms of the optimal choice of the so-called \emph{admissible} subsets of the function class, and the term $\bar{T}_{r/n}(f)$, even with a crude multiplicative factor, is not guaranteed to appear in the bound.

\subsection{Variations}
In this paper we only focus on formulating bounds for \eqref{eq:instance-dependent-formulation} via generic chaining with respect to the distribution of $X$. However, there are variations of the formulation \eqref{eq:instance-dependent-formulation}, as well as, approaches to obtain a tail bound that are worth mentioning.

\paragraph{Data-dependent bounds} Abstract measures of complexity of function classes are generally hard to approximate in application. Therefore, data-dependent tail bounds, such as those based on empirical Rademacher averages, are sometimes favored over the standard tail bounds. Using a secondary randomness incorporated in the empirical process of interest, usually through symmetrization with Rademacher random variables, the data-dependent bounds are obtained by conditioning on the observed samples $(X_i)_{i\in[n]}$ ---here and throughout we use $[n]$ to denote $\{1,2,\dotsc,n\}$. In our problem of interest, it suffices to define the data-dependent analog of $\bar{T}_r(f)$. In particular, for i.i.d Rademacher random variables $(\varepsilon_i)_{i\in[n]}$ which are independent of everything else, the goal is to find an upper tail bound for
\begin{align*}
	\tilde{Z}_{\left(X_i\right)_{i\in [n]}} & \defeq \sup_{f\in \mc F} \frac{1}{n}\sum_{i=1}^n \varepsilon_i f(X_i) - \tilde{T}_{r/n,\left(X_i\right)_{i\in [n]}}(f)\,,
\end{align*}
where
\begin{align*}
	\tilde{T}_{r,\left(X_i\right)_{i\in [n]}}(f) & \defeq \inf_{\lambda\ge 0}\frac{r+ \E_n\left(\log \cosh\left(\lambda f(X)\right)\right)}{\lambda}\,,
\end{align*}
is a random seminorm. Observe that $\tilde{T}_{r,\left(X_i\right)_{i\in [n]}}(f) \le \sqrt{2r} \left(\E_n f^2(X)\right)^{1/2}$ due to the inequality $\log \cosh(z) \le z^2/2$. Therefore, with
\begin{align*}
	\mbb V(X_1,\dotsc,X_n) & \defeq\left\{\left(f(X_1),\dotsc,f(X_n)\right)\st f\in \mbb V\right\}\,,
\end{align*}
and
\begin{align*}
	{\mc F}(X_1,\dotsc,X_n) & \defeq\left\{\left(f(X_1),\dotsc,f(X_n)\right)\st f\in \mc F\right\}\,,
\end{align*}
the stochastic process $H(v) = n^{-1}\sum_{i=1}^n \varepsilon_i v_i$ over $v\in \mbb V(X_1,\dotsc,X_n)$ has sub-Gaussian increments (with respect to the randomness of the $\varepsilon_i$s). Then, it follows from Talagrand's \emph{majorizing measures theorem} \citep{Tal87} that $\sup_{v\in{\mc F}(X_1,\dotsc,X_n)} H(v)$ is essentially comparable to the \emph{Gaussian complexity} of ${\mc F}(X_1,\dotsc,X_n)$ which is a natural geometric quantity. The ``deflated'' version of the function class $\mc F$, that is introduced below in \ref{ssec:deflation}, can also be adapted similarly. 

\paragraph{Regularizing with $\bar{T}^2_{r/n}(f)$} Instead of finding an upper bound for \eqref{eq:instance-dependent-formulation}, we can similarly consider finding an upper bound for
\begin{align}
	Z' & = \sup_{f\in \mc F} \E_n f(X) - \frac{1}{2\lambda}\bar{T}_{r/n}^2(f)\,, \label{eq:instance-dependent-alt}
\end{align}
for a suitably chosen parameter $\lambda > 0$. As detailed below in \ref{ssec:deflation}, to derive a tail bound for $Z$ we introduce a contraction $A[\cdot]$ to deflate the function class $\mc F$ into $\mc A = \{f-A[f]\st f\in \mc F\}$. To obtain a tail bound for $Z'$, we modify the condition on $A[\cdot]$ to be a contraction, to
\begin{align*}
	\bar{T}_{\tilde{r}/n}(A[f]) & \le \bar{T}_{\tilde{r}/n}^2(f)/(2\lambda)\,, & \text{for all}\ f\in \mc F\,,
\end{align*}
for a certain $\tilde{r} = r + O(1)$. For functions $f\in \mc F$ for which $\bar{T}_{\tilde{r}/n}(f) \gg 2\lambda$, the inequality above is less restrictive than the corresponding inequality stated in \ref{thm:main-exponential}. Therefore, we afford to deflate the mentioned set of functions more aggressively. Of course, this gain comes at the cost of less deflation over the remainder of $\mc F$; the parameter $\lambda$ would allow us to trade-off and optimize the overall bound. Interestingly, $Z'$ is also related to the suprema of the normalized empirical process. In particular, we have
\begin{align*}
	Z' & \le \sup_{f\in \mc F}\sup_{\beta \in \mbb R}\ \left(\beta \E_n f(X) - \frac{1}{2\lambda}\bar{T}^2_{r/n}(\beta f)\right) \\
	   & = \frac{\lambda}{2}\sup_{f\in\mc F}\ \left(\frac{\E_n f(X)}{\bar{T}_{r/n}(f)}\right)^2 \,,
\end{align*}
where the second line follows from the fact that $\bar{T}_{r/n}(\cdot)$ is positive homogeneous. Therefore, we can view $2\lambda^{-1}\sup_{f\in \mc F} \E_n\left(\left(f- A[f]\right)(X)\right)$ as a proxy for the suprema of the (squared) normalized empirical process, i.e., $\sup_{f\in\mc F}\left(\E_n f(X)/\bar{T}_{r/n}(f)\right)^2$ which is usually bounded using the peeling method; see, e.g. \cite[Theorem 3.3]{BBM05}, where peeling is used in derivation of generalization bounds in terms of \emph{local Rademacher complexity}. Furthermore, applying symmetrization and using the data-dependent framework as discussed above, the formulation \eqref{eq:instance-dependent-alt} basically converts to the \emph{offset Rademacher complexity} \citep{LRS15} which is developed to derive localized uniform tail bounds without the restrictive boundedness conditions required by the standard localized bounds.

\section{Tail Bounds Assuming Finite Exponential Moments}\label{sec:main-exponential}
We basically follow the \emph{generic chaining} argument \citep{Tal14} with an initial deflation of the function class that enables us to achieve the instance-dependence we aimed for. Furthermore, we use a ``truncated chain'' in our derivations similar to the approach of \citet[Theorem 3.2]{Dir15}, with the distinction that we derive the tail bounds directly without resorting to the polynomial moments as in \citep{Dir15}.

\subsection{A Generalized $\gamma$ functional}
Let us define $\varrho_r(g,h) = \bar{T}_r(g-h)$ as a distance between a pair of functions $g,h\in \mbb V$. With notation overloading, we also denote the distance of a function $g\in \mbb V$ to a set of functions $\mc H\subseteq \mbb V$ by
\begin{align}
	\varrho_r(g,\mc H) & =\inf_{h\in \mc H}\bar{T}_r(g-h)\,.\label{eq:envelope-metric}
\end{align}
Similar to the truncated variant of Talagrand's $\gamma$ functionals introduced in \citep{Dir15}, for $\mc A\subseteq \mbb V$, and $\ubar{\ell}\in \mbb Z_{\ge 0}$ we define
\begin{align}
	\gamma(\mc A;\, r, \ubar{\ell},n) & = \inf_{\left(\mc A_i\right)_{i \ge 0}} \sup_{a\in \mc A}\ \sum_{\ell \ge \ubar{\ell}}\varrho_{(r + (r+1)2^{\ell-\ubar{\ell}})/n}(a, \mc A_\ell)\,,\label{eq:gamma}
\end{align}
where the infimum is taken over an increasing \emph{admissible} sequence $(\mc A_i)_{i\ge 0}$ of the subsets of $\mc A$ with $|\mc A_i|\le 2^{2^i}$ for $i\ge 1$, and  $\mc{A}_0= \{0\}$. For $n=1$, $\ubar{\ell} \approx \log_2(r)$, and the approximation $\bar{T}_r(g) \le r^{1/\alpha}\norm{g}_{\psi_\alpha}$ with $\norm{\cdot}_{\psi_\alpha}$ being a $\psi_\alpha$ Orlicz norm, defined below in \ref{apx-sec:Orlicz}, the $\gamma$ functional defined in \eqref{eq:gamma} effectively reduces to the Talagrand's (truncated) $\gamma_\alpha$ functional.  For a set $\mc A$, the Talagrand's $\gamma_\alpha$ functional with respect to the suitable pseudometric $\rho$ is defined as
\begin{align*}
	\gamma_\alpha(\mc A, \rho;\,\ubar{\ell}) & = \inf_{\left(\mc A_i\right)_{i\ge 0}} \sup_{a\in \mc A}\sum_{\ell\ge \ubar{\ell}}^\infty 2^{\ell/\alpha}\rho(a,\mc A_\ell)\,,
\end{align*}
where the infimum is again taken with respect to a sequence of admissible sets $(\mc A_i)_{i\ge 0}$. The importance of these types of functionals was first revealed by Talagrand's majorizing measures theorem \citep{Tal87}, whose appellation is due to the following essentially equivalent definition of $\gamma_\alpha(\mc A, \rho) = \gamma_\alpha(\mc A, \rho; 0)$:
\begin{align*}
	\gamma_\alpha(\mc A, \rho) & = \inf_{\mu}\sup_{a\in \mc A} \int_0^\infty \left(\log\frac{1}{\mu\left(\{b\in\mc A\st \rho(b,a)\le \varepsilon\}\right)}\right)^{1/\alpha}\d\varepsilon\,,
\end{align*}
with the infimum taken over probability measures $\mu$ on $\mc A$ \citep{Tal01}. The majorizing measures theorem confirms a conjecture due to \citet{Fer75} that the expectation of the supremum of the centered Gaussian process indexed by $\mc A$, is equivalent to $\gamma_2(\mc A,\rho)$ up to constant factors, with $\rho$ being the canonical pseudometric induced by the Gaussian process.

Evaluating or even finding a good approximation for a $\gamma$ functional of a \emph{general} set $\mc A$ can be challenging \citep{Tal01,vHan18a}, and the only solution could be ``guessing'' an appropriate majorizing measure or an admissible sequence of subsets \citep{Tal01}. By pulling the supremum into the summation in the definition of $\gamma_\alpha$ functional, the infimum over the admissible sets would be achieved with each $\mc A_i$ being a covering set of $\mc A$ of cardinality $2^{2^i}$. This approximation describes the \emph{Dudley's (entropy) integral inequality}  (see, e.g., \cite[Theorem 8.1.3]{Ver18}, \cite[equation 2.3]{Dir15}), i.e.,
\begin{align*}
	\gamma_\alpha(\mc A, \rho) \lesssim_\alpha \int_0^\infty \left(\log N(\mc A, \rho,\varepsilon)\right)^{1/\alpha}\d\varepsilon\,,
\end{align*}
where $N(\mc A,\rho, \varepsilon)$ is the covering number of $\mc A$ with respect to $\rho$-balls of radius $\varepsilon$, and $\lesssim_\alpha$ is the usual inequality sign up to a (positive) constant factor depending only on $\alpha$. If accurate estimates of the covering numbers of $\mc A$ are available, approximations of $\gamma_\alpha$ through Dudley's inequality are easy to compute. However, Dudley's inequality may not deliver sufficiently sharp approximations (see, e.g., \cite[Section 3.1]{vHan18a}). The notable approach of \citet{vHan18a} improves on Dudley's inequality by replacing the entropy numbers of the entire set $\mc A$ by those of certain scale-dependent ``thin'' subsets of $\mc A$, imitating the multiscale form of $\gamma_\alpha$. These thin subsets are ``smoothed projections'' of $\mc A$ expressed by minimizers of interpolation of the base metric and a given nonnegative functional at different scales \cite[Section 2.1]{vHan18a}. The resulting approximation of $\gamma_\alpha$ is shown to be sharp in several nontrivial examples where Dudley's inequality yields rather loose approximations \cite[Section 3]{vHan18a}.

It is worth mentioning that the $\gamma$ functional defined by \eqref{eq:gamma} applies in more general settings than the standard $\gamma_\alpha$ functionals thanks to the less restricted form of the dependence of the pseudometric $\varrho_r(\cdot,\cdot)$ on the ``resolution scale'' $r$. If $\mc A$, the function class of interest, is inhomogeneous in the sense that it contains functions with significantly different tail behavior, then the standard $\gamma_\alpha$ functionals might overestimate the size (or complexity) of $\mc A$. As an illustrative example, suppose that for some absolute constant $\eta>0$ we have $\bar{T}_r(f) \approx \norm{f}_{L_\infty}r + \eta\norm{f}_{L_2}\sqrt{r}$ for all $f\in \mbb V$, where approximation is in a multiplicative sense, and $L_\infty$ and $L_2$ norms are defined with respect to the law of $X$. This form of dependence on the resolution scale cannot be reproduced by the $\gamma_\alpha$ functional or other similarly defined quantities where the resolution scale and the distance to an admissible set are decoupled. Measuring the distance with respect to the scale-insensitive norm $\norm{f} = c_\infty\norm{f}_{L_\infty} + c_2\norm{f}_{L_2}$ for arbitrary absolute constants $c_2, c_\infty \ge 0$, leads to a suboptimal upper bound $\bar{T}_r(f) \le (c'_\infty r + c'_2\sqrt{r})(c_\infty\norm{f}_{L_\infty} + c_2\norm{f}_{L_2})$ with $c'_2 , c'_\infty>0$ being constants that may depend only on $\eta$.

\subsection{Generic Chaining with a ``Deflation'' Step}\label{ssec:deflation}
The following theorem is our first main result.

\begin{thm}\label{thm:main-exponential} Let $A\st \mc F \to \mc F$ be a mapping such that
	\begin{align*}
		\bar{T}_{(r+k)/n}(A[f]) & \le \bar{T}_{(r+k)/n}(f)\,,\quad \text{for all}\ f\in \mc F\,,
	\end{align*}
	and
	\begin{align*}
		|A[\mc F]| & \le e^k\,,
	\end{align*}
	for some nonnegative integer $k$, where $A[\mc F] = \left\lbrace A[f]\st f\in \mc F\right\rbrace$ denotes the range of $A[\cdot]$. Furthermore, denote the ``deflation'' of $\mc F$ induced by $A[\cdot]$ by
	\begin{align*}
		\mc A & =\left\lbrace f - A[f]\st f\in \mc F\right\rbrace\,.
	\end{align*}
	Setting $\ubar{\ell}=\lfloor\log_2(r/3)\rfloor$ for $r\ge \log (2)$,
	with probability at least $1-2e^{-r}$, for all $f\in \mc F$ we have
	\begin{align}
		\E_n f(X) - \bar{T}_{(r+k)/n}(f) & \le  2\gamma(\mc A;\, r, \ubar{\ell},n) + \min\left\{2\bar{T}_{(2r+1)/n}(f- A[f]),\, \rad_{(2r+1)/n}(\mc A)\right\}\,, \label{eq:instance-dependent}
	\end{align}
	where $\gamma(\mc A;\, r, \ubar{\ell}, n)$ is defined as in \eqref{eq:gamma}, and $\rad_s(\mc A)=\max_{a\in \mc A} \bar{T}_s(a)$ denotes the radius of $\mc A$ measured by the seminorm $T_s(\cdot)$. The bound can further be optimized with respect to the mapping $A[\cdot]$, which both $k$ and $\mc A$ depend on.
\end{thm}

Let us pause here to make a few remarks about \ref{thm:main-exponential}. First, by taking the supremum with respect to $f\in \mc F$ on the right-hand side of \eqref{eq:instance-dependent} we obtain the following simplified version of the bound
\begin{align}
	\E_n f(X) - \bar{T}_{(r+k)/n}(f) & \le  2\gamma(\mc A;\, r, \ubar{\ell},n) + \rad_{(2r+1)/n}(\mc A)\,. \label{eq:instance-dependent-simplified}\tag{\ref{eq:instance-dependent}\textdegree}
\end{align}
This simplification is innocuous for worst-case choices of $f\in \mc F$ for which the term $2\bar{T}_{(2r+1)/n}(f- A[f])$ can be as large as \emph{twice} $\rad_{(2r+1)/n}(\mc A)$. However, we prefer the more general bound \eqref{eq:instance-dependent} over \eqref{eq:instance-dependent-simplified}, because $2\bar{T}_{(2r+1)/n}(f- A[f])$ can be much smaller than $\rad_{(2r+1)/n}(\mc A)$ for most choices of $f\in \mc F$, which may be useful in some applications. Second, the effectiveness of the deflation step becomes clear by observing that the result of the standard generic chaining argument can be reproduced by the possibly suboptimal choice of $A[f]=0$ for all $f\in \mc F$ in \eqref{eq:instance-dependent}. The admissible sequence in a standard generic chaining argument must cover $\mc F$, whereas in our formulation the admissible sequence must cover $\mc A$, the deflated version of $\mc F$. In particular, a desirable situation occurs when we can choose $A[\cdot]$ with $k\ll r$ such that $\gamma(\mc A;\,r,\ubar{\ell},n)$ is smaller than $\gamma(\mc F;\,r, 0,n)$, and $\bar{T}_{(r+k)/n}(f)$ is close to $\bar{T}_{r/n}(f)$. Third, the assumption that $A[\mc F]$ is finite, is not essential; as can be seen below in the proof, it suffices to guarantee that $\E_n A[f](X) \le \bar{T}_{(r+O(1))/n}(f)$ holds, with probability at least  $1 - e^{-r}$, for all $f\in \mc F$. For example, in \ref{prop:Gaussian-instance-dependent-tail}, this condition is shown to hold through the Gaussian concentration inequality. Finally, the right-hand side of \eqref{eq:instance-dependent} is basically an upper bound for $\sup_{a\in \mc A}\E_n a(X)$ that holds with probability at least $1-e^{-r}$. There are a few techniques to derive such upper bounds other than the generic chaining technique that we considered, such as \emph{Dudley's entropy integral} and the \emph{PAC-Bayesian argument} (see, \citep{AB07} for a shortlist of the different techniques). The generic chaining has the advantage that it applies under rather general conditions, and in the case of Gaussian processes (as in the example of \ref{sec:Examples}) and certain other families of distributions (see \citep{LT15} and references therein), yields sharp bounds.

\begin{proof}[Proof of \ref{thm:main-exponential}]
	As in \eqref{eq:gamma}, let $(\mc A_\ell)_{\ell\ge 0}$ be an increasing admissible sequence of subsets of $\mc A$ such that $\mc A_0=\{0\}$. Let $c>0$ denote a constant that we will specify later in the proof, and set $r_\ell = r + (r+c)2^{\ell- \ubar{\ell}}$. Given $A[\cdot]$ and the sequence $(\mc A_\ell)_{\ell\ge 0}$, we can decompose every $f\in \mc F$ as
	\begin{align*}
		f & = A[f] + f - A[f]                                                                        \\
		  & =  A[f] +  A_{\ubar{\ell}}[f]+\sum_{\ell \ge \ubar{\ell}} (A_{\ell+1}[f] - A_\ell[f])\,,
	\end{align*}
	where $A_\ell[f]$ denotes a function in $\mc A_\ell$ that is closest to $f-A[f]$  with respect to the seminorm $\bar{T}_{r_\ell}(\cdot)$, i.e.,
	\begin{align*}
		A_\ell[f] & = \argmin_{a\in \mc A_\ell} \bar{T}_{r_\ell/n}(f-A[f]-a)\,.
	\end{align*}
	It follows from the above decomposition that
	\begin{equation}
		\begin{aligned}
			\E_n f(X) - \bar{T}_{(r+k)/n}(f) & = \E_n A[f](X) - \bar{T}_{(r+k)/n}(f) +                                                            \\
			                                 & \qquad \E_n A_{\ubar{\ell}}[f](X)+ \sum_{\ell\ge \ubar{\ell}} \E_n (A_{\ell+1}[f]-A_\ell[f])(X)\,.
		\end{aligned}\label{eq:decomposition}
	\end{equation}
	\ref{lem:Chernoff} and a simple union bound guarantee that, with probability at least $1-|A[\mc F]|e^{-r-k}\ge 1-e^{-r}$, we have
	\begin{align}
		\E_n A[f](X) & \le \bar{T}_{(r+k)/n}(A[f]) \,,\quad\text{for all}\ f\in \mc F\,.\label{eq:uniform-A}
	\end{align}
	Similarly, with probability at least $1-2^{2^{\ubar{\ell}}}e^{-r_{\ubar{\ell}}}$, we have
	\begin{align}
		\E_n A_{\ubar{\ell}}[f](X) & \le \bar{T}_{r_{\ubar{\ell}}/n}(A_{\ubar{\ell}}[f])\,,\quad\text{for all}\ f\in \mc F\,.\label{eq:uniform-A-ell}	\end{align}

	Furthermore, for each index $\ell\ge \ubar{\ell}$ there are at most $|\mc A_{\ell+1}|\,|\mc A_\ell|\le 2^{2^{\ell+1}}2^{2^\ell}\le 2^{2^{\ell+2}}$ different functions $A_{\ell+1}[f]-A_\ell[f]$ as $f$ varies in $\mc F$. Applying \ref{lem:Chernoff} and the union bound again it follows that, with probability at least $1-2^{2^{\ell+2}}e^{-r_\ell}$, we also have
	\begin{align}
		\E_n \left(A_{\ell+1}-A_\ell\right)[f](X) & \le \bar{T}_{r_{\ell}/n}(\left(A_{\ell+1}-A_\ell\right)[f])\quad\text{for all}\ f\in \mc F\,.
		\label{eq:uniform-delta-A}
	\end{align}
	Putting \eqref{eq:uniform-A}, \eqref{eq:uniform-A-ell}, and \eqref{eq:uniform-delta-A} back in the decomposition \eqref{eq:decomposition}, with probability at least $1-e^{-r}-2^{2^{\ubar{\ell}}}e^{-r_{\ubar{\ell}}}-\sum_{\ell\ge \ubar{\ell}} 2^{2^{\ell+2}}e^{-r_\ell}$, we have
	\begin{align*}
		 & \E_n f(X) - \bar{T}_{(r+k)/n}(f)                                                                                                                                                           \\
		 & \le \bar{T}_{(r+k)/n}(A[f]) - \bar{T}_{(r+k)/n}(f) + \bar{T}_{r_{\ubar{\ell}/n}}(A_{\ubar{\ell}}[f]) + \sum_{\ell\ge\ubar{\ell}}\bar{T}_{r_\ell/n}\left(A_{\ell+1}[f] - A_{\ell}[f]\right) \\
		 & \le  2\bar{T}_{r_{\ubar{\ell}/n}}(f- A[f]) + \sum_{\ell\ge\ubar{\ell}}\bar{T}_{r_\ell/n}\left(A_{\ell+1}[f] - A_{\ell}[f]\right)                                                           \\
		 & \le 2\bar{T}_{r_{\ubar{\ell}/n}}(f- A[f]) + \sum_{\ell\ge\ubar{\ell}}\bar{T}_{r_{\ell}/n}\left(f - A[f] - A_{\ell}[f]\right)+\bar{T}_{r_{\ell}/n}\left(f - A[f] - A_{\ell+1}[f]\right)     \\
		 & \le 2\bar{T}_{r_{\ubar{\ell}/n}}(f- A[f])+ 2\sum_{\ell\ge\ubar{\ell}}\bar{T}_{r_\ell/n}\left(f - A[f] - A_{\ell}[f]\right)\,,\quad \text{for all}\ f \in \mc F\,,
	\end{align*}
	where the second inequality follows from the assumption $\bar{T}_{(r+k)/n}(A[f]) \le \bar{T}_{(r+k)/n}(f)$ and the fact that
	\begin{align}
		\bar{T}_{r_{\ubar{\ell}/n}}(A_{\ubar{\ell}}[f]) & \le \bar{T}_{r_{\ubar{\ell}/n}}(f - A[f] - A_{\ubar{\ell}}[f]) + \bar{T}_{r_{\ubar{\ell}/n}}(f- A[f]) \nonumber \\
		                                                & \le 2\bar{T}_{r_{\ubar{\ell}/n}}(f- A[f])\,, \label{eq:residual}
	\end{align}
	and the third and fourth inequalities respectively follow from part \labelcref{itm:envelope} of \ref{lem:T-r(g)} and the fact that $\bar{T}_r(f)$ inherits the monotonicity with respect to $r$ from $T_r(f)$. Recalling the definition \eqref{eq:envelope-metric}, on the same event we can write
	\begin{align*}
		\E_n f(X) - \bar{T}_{(r+k)/n}(f) - 2\bar{T}_{r_{\ubar{\ell}/n}}(f- A[f]) & \le 2\sum_{\ell\ge\ubar{\ell}}\varrho_{r_{\ell}/n}\left(f - A[f], \mc A_{\ell}\right)\,,\quad \text{for all}\ f \in \mc F \\
		                                                                         & \le 2\sup_{a\in \mc A}\sum_{\ell\ge\ubar{\ell}}\varrho_{r_\ell/n}\left(a, \mc A_{\ell}\right)\,.
	\end{align*}
	Taking the infimum with respect to the admissible subsets $(\mc A_i)_{i\ge 0}$ on the right-hand side yields
	\begin{align*}
		\E_n f(X) - \bar{T}_{(r+k)/n}(f) - 2\bar{T}_{r_{\ubar{\ell}/n}}(f- A[f]) & \le 2\gamma(\mc A;\,r,\ubar{\ell},n)\,.
	\end{align*}
	Furthermore, if instead of the inequality \eqref{eq:residual} we use
	\begin{align*}
		\sup_{f\in \mc F} \bar{T}_{r_{\ubar{\ell}}/n}(A_{\ubar{\ell}}[f]) & \le \rad_{r_{\ubar{\ell}}/n}(\mc A) = \sup_{a\in \mc A} \bar{T}_{r_{\ubar{\ell}}/n}(a)\,,
	\end{align*}
	the corresponding terms $2\bar{T}_{r_{\ubar{\ell}/n}}(f- A[f])$ in the subsequent inequalities can all be replaced by $\rad_{r_{\ubar{\ell}}/n}(\mc A)$. Then \eqref{eq:instance-dependent} follows as the better of the two resulting bounds.

	To complete the proof, it suffices to show that for $c=\log(2)<1$, and the prescribed $\ubar{\ell}=\left\lfloor \log_2(r/3)\right\rfloor$,  we have $2^{2^{\ubar{\ell}}}e^{-r_{\ubar{\ell}}} + \sum_{\ell\ge\ubar{\ell}}2^{2^{\ell+2}}e^{-r_{\ell}}\le e^{-r}$. The specific choices of $c$ and $\ubar{\ell}$ ensures that for $\ell\ge \ubar{\ell}$ we have
	\begin{align*}
		2^{2^{\ell+2}} e^{-(r_{\ell}-r)} & = \left(2^{2^{\ubar{\ell}+2}}e^{-r-c}\right)^{2^{\ell-\ubar{\ell}}} \\
		                                 & \le 2^{-(2^{\ell - \ubar{\ell}})}\,.
	\end{align*}
	Furthermore, we have
	\begin{align*}
		2^{2^{\ubar{\ell}}}e^{-(r_{\ubar{\ell}}-r)} & \le \frac{1}{16}\,.
	\end{align*}
	The desired inequality for the tail probability then follows as
	\begin{align*}
		2^{2^{\ubar{\ell}}}e^{-r_{\ubar{\ell}}} + \sum_{\ell\ge\ubar{\ell}}2^{2^{\ell+2}}e^{-r_{\ell}} & = \left(2^{2^{\ubar{\ell}}}e^{-(r_{\ubar{\ell}}-r)} + \sum_{\ell\ge\ubar{\ell}}2^{2^{\ell+2}}e^{-(r_{\ell}-r)}\right) e^{-nr} \\
		                                                                                               & \le \left(\frac{1}{16} + \frac{1}{2} +  \sum_{\ell> \ubar{\ell}}2^{-2^{\ell-\ubar{\ell}}}\right) e^{-r}                       \\
		                                                                                               & < \left(\frac{9}{16} + \frac{1/4}{1-1/4}\right) e^{-r}                                                                        \\
		                                                                                               & < e ^{-r}\,.
	\end{align*}
\end{proof}

\section{Examples}\label{sec:Examples}
In this section we consider two examples to further expose the structure and utility of instance-dependent bounds, and show that \ref{thm:main-exponential} provides optimal or nearly-optimal bounds. We show that, up to constant factors, \ref{thm:main-exponential} reproduces the bounds provided below in \ref{prop:Gaussian-instance-dependent-tail,prop:CI-Gaussian-top-m}. Proofs of these propositions as stated are also provided in \ref{apx-sec:proofs}.

\subsection{Marginals of a Gaussian Vector}\label{ssec:Gaussian-marginals}
We first consider the case where the function class consists of \emph{linear} functionals indexed by the centered unit Euclidean ball, i.e.,
\begin{align*}
	\mc{F} = \lbrace x\mapsto\inp{u,x}\st\, u\in\mbb R^d\,, \norm{u}_2\le 1\rbrace\,,
\end{align*}
and the law of the underlying random variable $X\in \mbb R^d$ is $\mr{Normal}(0,\Sigma)$. This scenario is studied in \citep{LM23} who established the following proposition. The original statement in \cite{LM23} uses slightly different formulation and notation. For example, the terms $N$, $\sigma(u)$, and $\log(1/\delta)$ in the original notation respectively correspond to $n$, $(u^\T\Sigma u)^{1/2}$, and $r$ in our formulation. Furthermore, \cite{LM23} considers a scaled version of the deviation term $\sqrt{2r/n}(u^\T\Sigma u)^{1/2}$ and effectively analyzes the upper and lower bounds for $\sup_{u\st \norm{u}_2\le 1} \inp{u,X} - C\sqrt{2r/n}(u^\T\Sigma u)^{1/2}$ for some absolute constant $C>0$. This can be reproduced in our formulation by specializing to $k = cr$ for some constant $c>0$.

\begin{prop}[{\citet[Proposition 1]{LM23}}]
	\label{prop:Gaussian-instance-dependent-tail}
	Let $X\sim \mr{Normal}(0, \Sigma/n)$ be a random vector in $\mbb R^d$, and denote the eigenvalues of the (scaled) covariance matrix $\Sigma$ by $\lambda_1\ge \lambda_2 \ge \dotsb \ge \lambda_d$. Furthermore, let
	\begin{align*}
		S_k & = \sup_{u\in \mbb R^d \st\norm{u}_2 \le 1}\E_n \inp{u,X} - \frac{\sqrt{2r}+\sqrt{k}}{\sqrt{n}}(u^\T\Sigma u)^{1/2}\,.
	\end{align*}
	Then, for any nonnegative integer $k\le d$, with probability at least $1-2e^{-r}$ we have\footnote{We treat the summations whose lower index is larger than their upper index as empty summations that evaluate to zero.}
	\begin{align}
		S_k & \le \sqrt{\frac{\sum_{i=k+1}^d \lambda_i}{n}}+  \sqrt{\frac{2r}{n}\lambda_{k+1}}\,.\label{eq:Gaussian-Prop-upperbound}
	\end{align}
	Furthermore, with $k' = \lceil 6r + 3(\sqrt{2r}+\sqrt{k})^2 \rceil$, with probability at least $1-2e^{-r}$ we have
	\begin{align}
		S_k & \ge \sqrt{\frac{\sum_{i=k'+1}^d \lambda_i}{3n}}\,.\label{eq:Gaussian-Prop-tightness}
	\end{align}
\end{prop}

To understand the significance of \ref{prop:Gaussian-instance-dependent-tail} as well as the role of the integer parameter $k$, it is worth comparing the derived instance-dependent bound to the conventional bounds. A standard approach to bound $\E_n \inp{u,X}$ uniformly for $\norm{u}_2\le 1$ is to apply the \emph{Gaussian concentration inequality} (see, e.g., \cite[Theorem 5.6]{BLM13}) to $\norm{\E_n X}_2$, which, with probability at least $1-e^{-r}$, guarantees that
\begin{equation*}
	\sup_{u\st \norm{u}_2 \le 1}\E_n \inp{u,X} = \norm{\E_n X}_2 \le \sqrt{\frac{\tr(\Sigma)}{n}} + \sqrt{\frac{2r}{n}\,\norm{\Sigma}_\mr{op}}\,,
\end{equation*}
where $\tr(\cdot)$ and $\norm{\cdot}_\mr{op}$, respectively, denote the trace and the operator norm of their matrix arguments. This bound pessimistically considers the worst-case deviation for all of the random variables $\E_n \inp{u,X}$. By setting $k=0$ in \eqref{eq:Gaussian-Prop-upperbound}, we can reproduce this pessimistic bound, except for an extra factor of $2$ in front of the $r$-dependent term. A much better choice for $k$ in the instance-dependent tail bound can be found as follows. For $\ell=0,1,\dotsc,d$, let $\Sigma_\ell$ denote the best rank-$\ell$ approximation of $\Sigma$ with respect to the operator norm, and denote the \emph{effective rank} of $\Sigma - \Sigma_\ell$ by  $d_\ell = \tr(\Sigma - \Sigma_\ell)/\norm{\Sigma-\Sigma_\ell}_\mr{op}$, with the convention that $0/0 = 1$ at $\ell=d$. Furthermore, define
\begin{align*}
	k_\* & =\argmin_{\ell=0,1,\dotsc,d} \max\left\{\frac{2r}{d_\ell},\frac{\ell}{2}\right\}\,,
\end{align*}
and
\begin{align*}
	C_\* & = 	\max\left\{\sqrt{\frac{2r}{d_{k_\*}}},\sqrt{\frac{k_\*}{2r}}\right\}\,.
\end{align*}
Then, setting $k=k_\*$ and straightforward manipulations of the tail bound in \eqref{eq:Gaussian-Prop-upperbound} yields the inequality
\begin{align*}
	\E_n \inp{u,X} &
	\le (C_\* +1) \left(\sqrt{\frac{\tr(\Sigma -\Sigma_{k_\*})}{n}} + \sqrt{\frac{2r}{n}}(u^\T\Sigma u)^{1/2}\right)\,.
\end{align*}
Since $r$ determines the confidence level of the tail bound,  $k_\*$ and $C_\*$ only depend on this confidence level and the spectral characteristics of $\Sigma$. A favorable situation occurs when $C_\*$ is a small constant, which requires both $k_\*$ and $d_{k_\*}$ to be proportional to $r$.

We provide a slightly different a more streamlined proof of \ref{prop:Gaussian-instance-dependent-tail} in the appendix that makes the constant factors reasonably small and explicit. Our proof only invokes the Gaussian concentration inequality, whereas the original proof in \citep{LM23} uses the Gaussian Poincar\'{e} inequality as well.

To put this special case in the general perspective, observe that, with $\mbb V$ being the set of linear functionals over $\mbb R^d$, the function class $\mc{F}$ consists of functions $f(x) = f_u(x)= \inp{u,x}$ with $\norm{u}_2\le 1$, and we have
\[ T_r(f) = \bar{T}_r(f)  = \sqrt{2r/n}\sqrt{u^\T\Sigma u}= \sqrt{2r/n}\norm{f}_{L_2(X)}\,.\]

\subsubsection{Reproducing \eqref{eq:Gaussian-Prop-upperbound} via \ref{thm:main-exponential}}\label{sssec:Gaussian-marginals-Thm1}
Let $\mc B_{k_0}$ denote the centered $k_0$-dimensional unit Euclidean ball in the span of the top $k_0$ eigenvectors of $\Sigma$, i.e., the column space of $\Sigma_{k_0}$. Furthermore, for a suitably small $\epsilon>0$ let $\mc N_{\epsilon/2}$ denote an $\epsilon/2$-net of $\mc B_{k_0}$ with respect to the norm $\norm{u}_\Sigma \defeq (u^\T\Sigma u)^{1/2}$. Then, for $f_u(\cdot) =  \inp{u,\cdot} \in \mc F$ we may choose
\begin{align*}
	A[f_u] & = f_{\hat{u}_{\epsilon}}\,,
\end{align*}
where
\begin{align*}
	\hat{u}_\epsilon & = \begin{cases}
		                     \left(1-\frac{\epsilon}{2\norm{\tilde{u}_{\epsilon/2}}_\Sigma}\right)_+\tilde{u}_{\epsilon/2} & \text{if}\ \norm{\tilde{u}_{\epsilon/2}}_\Sigma > \norm{u}_\Sigma\,, \\
		                     \tilde{u}_{\epsilon/2}                                                                        & \text{otherwise}\,,
	                     \end{cases}
\end{align*}
with
\begin{align*}
	\tilde{u}_{\epsilon/2} & = \argmin_{u'\in \mc N_{\epsilon/2}} \norm{u-u'}_\Sigma\,.
\end{align*}
This construction ensures that $\norm{u - \hat{u}_\epsilon}_\Sigma = \sqrt{\norm{u}_{\Sigma - \Sigma_{k_0}}^2 + \norm{u-\hat{u}_\epsilon}_{\Sigma_{k_0}}^2} \le \sqrt{\norm{u}_{\Sigma - \Sigma_{k_0}}^2 + \epsilon^2}$ and $\norm{\hat{u}_\epsilon}_\Sigma\le \norm{u}_\Sigma$. With the choices made so far, we have
\begin{align*}
	\mc A & = \left\{f_u - f_{\hat{u}_\epsilon} = \inp{u-\hat{u}_\epsilon,\cdot}\st \norm{u}_2 \le 1\right\}\,.
\end{align*}
With $\rho_2(x,y)\defeq \norm{x-y}_2$ denoting the normalized Euclidean metric, we have
\begin{align*}
	\gamma(\mc A ;r,\ubar{\ell},n) & \lesssim n^{-1/2}\gamma_2(\mc V_{k_0} + \mc V^\perp_{k_0,\epsilon},\rho_2;\ubar{\ell})                                   \\
	                               & \le n^{-1/2}\gamma_2(\mc V_{k_0},\rho_2;\ubar{\ell})+ n^{-1/2}\gamma_2(\mc V^\perp_{k_0,\epsilon},\rho_2;\ubar{\ell})\,,
\end{align*}
where
\begin{align*}
	\mc V_{k_0} & = \left\{(\Sigma -\Sigma_{k_0})^{1/2}u\st \norm{u}_2 \le 1\right\}\,,
\end{align*}
and
\begin{align*}
	\mc V^\perp_{k_0,\epsilon} & = \left\{\Sigma_{k_0}^{1/2}(u - \hat{u}_\epsilon)\st \norm{u}_2 \le 1\right\}\,.
\end{align*}
By the majorizing measures theorem \cite[Theorem 2.4.1]{Tal14}, with $Z\sim \mr{Normal}(0,I)$ we have
\begin{align*}
	\gamma_2(\mc V_{k_0},\rho_2;\ubar{\ell}) & \lesssim \E \sup_{v\in \mc V_{k_0}} \inp{v,Z} \\
	                                         & \le \E \norm{(\Sigma-\Sigma_{k_0})^{1/2}Z}_2  \\
	                                         & \le \sqrt{\sum_{i=k_0+1}^d \lambda_i}\,,
\end{align*}
and
\begin{align*}
	\gamma_2(\mc V^\perp_{k_0,\epsilon},\rho_2;\ubar{\ell}) & \lesssim \E \sup_{v\in \mc V^\perp_{k_0,\epsilon}} \inp{v,Z} \\
	                                                        & \le \E \sup_{v\in \epsilon \mc B_{k_0}} \inp{v,Z}            \\
	                                                        & \le \sqrt{k_0}\epsilon\,.
\end{align*}
Furthermore, we have
\begin{align*}
	\rad_{(2r+1)/n}(\mc A) & \le \sqrt{\frac{2r+1}{n}}(\sqrt{\lambda_{k_0+1}} + \epsilon)\,.
\end{align*}
With these bounds at hand, invoking \ref{thm:main-exponential} with $k\ge \log(|\mc N_{\epsilon/2}|)$ guarantees that with probability at least $1-2e^{-r}$ for every $u$ in the unit $\ell_2$ ball we have
\begin{align}
	\inp{u,X} - \sqrt{\frac{2(r+k)}{n}}\norm{u}_\Sigma & \lesssim  \frac{1}{\sqrt{n}}\left(\sqrt{\sum_{i=k_0+1}^d \lambda_i} + \sqrt{k_0}\epsilon \right)+ \sqrt{\frac{r}{n}}(\sqrt{\lambda_{k_0+1}} + \epsilon)\label{eq:bound-via-epsilon-net}
\end{align}
By a na\"ive approximation we have $|\mc N_{\epsilon/2}| \le \left(1 + 4\sqrt{\lambda_1}/\epsilon\right)^{k_0}$. Therefore, we must have $\epsilon \ge 4\sqrt{\lambda_1}/(2^{-1/k_0}e^{k/k_0}-1)$. In particular, if $k\ge k_0\log(1+4\sqrt{\lambda_1/\lambda_{2k_0}})$, then we can choose $\epsilon = \min\{\sqrt{\lambda_{k_0+1}}, \sqrt{\sum_{i>k_0}\lambda_i/k_0}\}$ and \eqref{eq:bound-via-epsilon-net} simplifies to
\begin{align*}
	\inp{u,X} - \sqrt{\frac{2(r+k)}{n}}\norm{u}_\Sigma & \lesssim  \sqrt{\frac{\sum_{i=k_0+1}^d \lambda_i}{n}}+ \sqrt{\frac{r}{n}}\sqrt{\lambda_{k_0+1}}\,,
\end{align*}
which, assuming that $\lambda_1/\lambda_{2k_0}$ is a constant, is effectively \eqref{eq:Gaussian-Prop-upperbound} up to the constant factors.

\subsection{Confidence Intervals for the ``Middle-Ranked'' Means of Correlated Gaussians}\label{ssec:CI-Gaussian-top-m}
In this subsection we derive confidence intervals for the $m$-th largest mean of correlated Gaussian random variables, as another example where instance-dependent tail bounds can be applied. The proof of \ref{prop:CI-Gaussian-top-m} provided in \ref{apx-sec:proofs}, again relies on the Gaussian concentration inequality, as well as a bound on the expected supremum of canonical Gaussian processes over (symmetric) polytopes \cite[Proposition 2.4.16 and Theorem 2.4.18]{Tal14} (see also the discussion in \cite[Section 3.3]{vHan18a}). These tools allow us to express the upper and lower bounds of the confidence interval in more explicit terms. We can basically recover \ref{prop:CI-Gaussian-top-m} through \ref{thm:main-exponential} as explained at the end of this subsection.

Our goal is to find an upper and lower bounds for the $m$-th largest entry of a parameter vector $\theta\in \mbb R^d$ for $m = o(d)$. We are only given $\hat{\theta} = \theta + X$, where $X$ is a zero-mean Gaussian random variable with covariance $\Sigma = \E XX^\T$. We assume that $\Sigma$ is known, and, without loss of generality, it is full-rank. For any vector $v$ we denote by $v^\downarrow$ the vector of the entries of $v$ sorted in decreasing order. Therefore, the $m$-th largest entry of a vector $v$ can be expressed as $v^\downarrow_m$. Furthermore, for any subset $S$ of $[d]$ let $v_{S}\in \mbb R^{|S|}$ denote the restriction of $v$ to the entries indexed by $S$. We also use the shorthand $\Sigma_S = \E X_S X_S^\T$, which is the same as $\Sigma$ restricted to the rows and columns in $S$. By $\binom{[d]}{\ell}$, we denote the set of subsets of $[d]$ of size $\ell$, and we write $\triangle^\ell$ to denote the unit simplex in $\mbb R^\ell$

Perhaps the simplest approach for our problem is to use the inequality \[\left| \theta^\downarrow_m - \hat{\theta}^\downarrow_m\right| \le \norm{\theta-\hat{\theta}}_\infty = \norm{X}_\infty\,,\] that suggests a confidence interval centered at the plug-in estimator $\hat{\theta}^\downarrow_m$ whose width is no less than  $2\norm{X}_\infty$. The Gaussian concentration inequality then guarantees that \[\norm{X}_\infty \le \E\norm{X}_\infty + \sqrt{2r}\max_{i\in [d]} \Sigma^{1/2}_{i,i}\,,\] with probability at least $1-e^{-r}$. Furthermore, we can bound $\E\norm{X}_\infty$, viewed as the expected supremum of a canonical Gaussian process over a (symmetric) polytope, using \cite[Proposition 2.4.16 and Theorem 2.4.18]{Tal14}. Denoting the $i$-th largest diagonal entry of $\Sigma$ by $\Sigma^\downarrow_{i,i}$, for some constant $C>0$ we have
\begin{align}
	\left|\theta^\downarrow_m - \hat{\theta}^\downarrow_m \right| & \le C \max_{i\in [d]} \sqrt{\Sigma^\downarrow_{i,i}\log(i+1)} + \sqrt{2r}\sqrt{\Sigma^\downarrow_{1,1}}\,.\label{eq:theta-variational-form-simple-CI}
\end{align}

Another related problem is the problem of \emph{multiple comparisons} in hypothesis testing \cite{Hsu96}, with the prevalent models assuming independent noise (i.e., a diagonal covariance matrix for Gaussian data). The basic idea is that $(X-\theta)^\T \Sigma ^{-1}(X-\theta)$ follows a $\chi^2(d)$ distribution and we can examine the \emph{least favorable configuration} that match the corresponding quantile of a $\chi^2(d)$ random variable to find the upper and lower threshold of the confidence interval for $\theta_m^\downarrow$. In the case of a diagonal $\Sigma$, this approach reduces to finding $\bar{\alpha}$ (resp. $\ubar{\alpha}$) such that $\min_{S\in \binom{[d]}{m}} \left(X_i - \bar{\alpha} \right)^2/\sigma_i^2$  (resp. $\min_{S\in \binom{[d]}{d-m+1}} \left(X_i - \ubar{\alpha} \right)^2/\sigma_i^2$) matches the $1-e^{-t}$ quantile of $\chi^2(d)$. Therefore, the confidence interval (and its length) do not have a sufficiently simple analytic expression and must be computed numerically.

Using the instance-dependent uniform tail bounds, we establish a confidence interval for $\theta^\downarrow_m$  that is more refined than \eqref{eq:theta-variational-form-simple-CI}. At the end of this subsection we explain how this proposition follows from \ref{thm:main-exponential} by modifying certain steps of the proof provided in the \ref{apx-sec:proofs}.
\begin{prop}\label{prop:CI-Gaussian-top-m}
	Let $\hat{\theta} = \theta + X$ be a noisy observation of a parameter $\theta\in \mbb R^d$ with $X\sim \mr{Normal}(0,\Sigma)$. Furthermore, let  $m=o(d)$ be a positive integer\footnote{The little $o$ notation means that $m/d \to 0$ as $d\to \infty$}, $r\in \mbb R_{\ge 0}$, and $k\le \min\{r,m\}$ be a nonnegative integer. For any nonempty set $S\subseteq [d]$ denote by $\Sigma_{S,k}$ the best rank-$k$ approximation of $\Sigma_S$ with respect to the operator norm, and define the vector $\sigma = \sigma(S,k)$ such that $\sigma_i = \sqrt{(\Sigma_S - \Sigma_{S,k})_{i,i}}$ for $i\in[|S|]$. Then, defining
	\begin{align*}
		\sigma^*(S,k) & \defeq \max_{i\in [|S|]} \sigma^\downarrow_i \sqrt{\log(i+1)}\,,
	\end{align*}
	\begin{align*}
		Q_{S,\beta}(\vartheta) & = \max_{u\in \triangle^{|S|}} \inp{u,\vartheta} - \beta\norm{u}_{\Sigma_S}\,,
	\end{align*}
	which implicitly depends on $\Sigma_S$, and
	\begin{align*}
		\beta_{r,m,k} & =\sqrt{r+m + m\log(d/m) + k}\,,
	\end{align*}
	with probability at least $1-4e^{-r}$, for some some universal constant $C>0$ we have
	\begin{align}
		\theta^\downarrow_m & \ge \min_{S\in \binom{[d]}{d-m+1}} Q_{S,\beta_{r,m,k}}(\hat{\theta}_S) - C \sigma^*(S,k) - \sqrt{2}\beta_{r,m,k}\norm{\sigma(S,k)}_\infty\,, \label{eq:CI-lower}
	\end{align}
	and
	\begin{align}
		\theta^\downarrow_m & \le  \max_{S\in \binom{[d]}{m}} - Q_{S,\beta_{r,m,k}}(-\hat{\theta}_S) + C \sigma^*(S,k) + \sqrt{2}\beta_{r,m,k}\norm{\sigma(S,k)}_\infty\,.\label{eq:CI-upper}
	\end{align}
\end{prop}

If in addition to the assumption $m = o(d)$, we have $m\lesssim r$ (i.e., $m\le cr$ for some fixed constant $c>0$), the bounds above reproduce \eqref{eq:theta-variational-form-simple-CI} up to an extra logarithmic factor for the term $\sqrt{\Sigma^\downarrow_{1,1}}$.

We also have the following minimax lower bounds for estimating $\theta^\downarrow_m$, whose proof is provided in \ref{apx-sec:proofs}.
\begin{prop}\label{prop:lower-bound}
	For $\kappa\ge 0$, let $\Theta=\{\theta\in\mbb R^d\st \norm{\Sigma^{-1/2}\theta}_2 \le \kappa\}$ be a compact domain of parameters. With $\delta_m\ge 0$ defined as
	\begin{align*}
		\delta_m & = \sup_{\theta\in\Theta}\theta^\downarrow_m + \sup_{\eta\in \Theta}\eta^\downarrow_{d-m+1}\,.
	\end{align*}
	For any estimator $g(\hat{\theta})$ of $\theta^\downarrow_m$ we have
	\begin{align}
		\sup_{\theta \in \Theta} \E (g(\hat{\theta})-\theta^\downarrow_m)^2 & \ge \frac{\delta_m^2}{8e\max\{1,2\kappa^2\}} \,.\label{eq:expected-LB}
	\end{align}
	Furthermore, we have
	\begin{align*}
		\sup_{\theta\in \Theta}\P\left(|g(\hat{\theta}) - \theta^\downarrow_m| > \frac{\delta_m}{3\max\{1,\sqrt{2}\kappa\}}\right) & \ge \frac{1}{2e}\,.
	\end{align*}
\end{prop}

Because of the complicated and implicit form of the expressions in \eqref{eq:CI-lower} and \eqref{eq:CI-upper}, it is difficult to compare---in full generality---the width of the confidence interval provided by \ref{prop:CI-Gaussian-top-m} and the minimax lower bound of \ref{prop:lower-bound}. We only focus on the special case where $\Sigma$ is diagonal.	Furthermore, for the sake of simpler calculations we use the lower bound
\begin{align*}
	Q_{S,\beta}(\hat{\theta}_S) & \ge \max_{i\in S} \left(\hat{\theta}_i - \beta \sqrt{\Sigma_{i,i}}\right)\,.
\end{align*}
The width of the confidence interval expressed by \eqref{eq:CI-lower} and \eqref{eq:CI-upper}, which we denote by $\Delta_m$, can be bounded as
\begin{align*}
	\Delta_m = \max_{S\in \binom{[d]}{m}\,, S'\in \binom{[d]}{d-m+1}}                       & \Big(- Q_{S,\beta_{r,m,k}}(-\hat{\theta}_S) - Q_{S',\beta_{r,m,k}}(\hat{\theta}_{S'})   \\
	                                                                                        & + C (\sigma^*(S,k) + \sigma^*(S',k))                                                    \\
	                                                                                        & + \sqrt{2}\beta_{r,m,k}(\norm{\sigma(S,k)}_\infty+\norm{\sigma(S',k)}_\infty)\Big)      \\
	\le \max_{S\in \binom{[d]}{m}\,, S'\in \binom{[d]}{d-m+1}} \Big(\min_{i\in S\,,j\in S'} & \hat{\theta}_i -\hat{\theta}_j + \beta_{r,m,k}(\sqrt{\Sigma_{i,i}}+\sqrt{\Sigma_{j,j}}) \\
	                                                                                        & + C (\sigma^*(S,k) + \sigma^*(S',k))                                                    \\
	                                                                                        & + \sqrt{2}\beta_{r,m,k}(\norm{\sigma(S,k)}_\infty+\norm{\sigma(S',k)}_\infty)\Big)      \\
	\le \max_{S\in \binom{[d]}{m}\,, S'\in \binom{[d]}{d-m+1}} \Big(\min_{i\in S\cap S'}    & 2\beta_{r,m,k}\sqrt{\Sigma_{i,i}}  + C (\sigma^*(S,k) + \sigma^*(S',k))                 \\
	                                                                                        & + \sqrt{2}\beta_{r,m,k}(\norm{\sigma(S,k)}_\infty+\norm{\sigma(S',k)}_\infty)\Big)\,,
\end{align*}
where the second inequality holds because $|S\cap S'| = |S|+|S'|-|S\cup S'| \ge 1$, and we can choose $i=j\in S\cap S'$. Furthermore, we have the inequalities
\begin{align*}
	\max\left\{\norm{\sigma(S,k)}_\infty, \norm{\sigma(S',k)}_\infty\right\} & \le \sqrt{\Sigma^\downarrow_{1,1}}\,,
\end{align*}
and
\begin{align*}
	\max\left\{\sigma^*(S,k),\sigma^*(S',k)\right\} & \le \max_{i\in [d]}\sqrt{\Sigma^\downarrow_{i+k,i+k}\log(i+1)}\,,
\end{align*}
using which we deduce
\begin{align*}
	\Delta_m & \le 2C\max_{i\in [d-k]} \sqrt{\Sigma^\downarrow_{i+k,i+k}\log(i+1)} + (2+2\sqrt{2})\beta_{r,m,k}\sqrt{\Sigma^\downarrow_{1,1}}\,.
\end{align*}

With $\Theta$ defined as in \ref{prop:lower-bound} we have
\begin{align*}
	\sup_{\theta\in \Theta} \theta^\downarrow_m & = \left(\sum_{i=1}^m\frac{1}{\Sigma^\downarrow_{i,i}}\right)^{-1/2}\kappa\,,
\end{align*}
and
\begin{align*}
	\sup_{\eta\in\Theta} \eta^\downarrow_{d-m+1} = \left(\sum_{i=1}^{d-m+1}\frac{1}{\Sigma^\downarrow_{i,i}}\right)^{-1/2}\kappa\,.
\end{align*}
Therefore, \ref{prop:lower-bound} implies that any confidence interval for $\theta^\downarrow_m$ with coverage probability no less than $1-1/(2e)$, should have a width equal to $C_\kappa\left(\sum_{i=1}^m\frac{1}{\Sigma^\downarrow_{i,i}}\right)^{-1/2}$ for some constant $C_\kappa\ge 0$ that may depend on $\kappa$. In particular, for any $\theta \in \Theta$ we have
\begin{align*}
	\P\left(\Delta_m > C_\kappa\left(\sum_{i=1}^m\frac{1}{\Sigma^\downarrow_{i,i}}\right)^{-1/2} \right) & \ge \frac{1}{2e}\,.
\end{align*}
Choosing $r = 1+\log(8)\approx 3$, we have also shown that
\begin{align*}
	\P\left(\Delta_m > 2C\max_{i\in [d-k]} \sqrt{\Sigma^\downarrow_{i+k,i+k}\log(i+1)} + (2+2\sqrt{2})\beta_{r,m,k}\sqrt{\Sigma^\downarrow_{1,1}} \right) & \le \frac{1}{2e}\,.
\end{align*}
Then, if we define
\begin{align*}
	p_m(\Sigma) & \defeq \sum_{i=1}^m \frac{\Sigma^\downarrow_{1,1}}{\Sigma^\downarrow_{i,i}}\,,
\end{align*}
and
\begin{align*}
	q_{m,k}(\Sigma) & \defeq \max_{i\in [d-k]} \frac{\Sigma^\downarrow_{i+k,i+k}}{\Sigma^\downarrow_{1,1}}\log(i+1)  \,,
\end{align*}
then $\Delta_m$ is optimal up to a factor $\mr{polylog}(d)$,  if $m$, $p_m(\Sigma)$, and $q_{m,k}(\Sigma)$ are all bounded from above as $\mr{polylog}(d)$. Specifically, if $m$, $p_m(\Sigma)$, and $q_{m,k}(\Sigma)$ are all absolute constants, then $\Delta_m$ is optimal up to a constant factor.

\subsubsection{Reproducing \eqref{eq:CI-lower} and \eqref{eq:CI-upper} via \ref{thm:main-exponential}}\label{sssec:middle-ranked-Thm1}
Proof of \eqref{eq:CI-lower} provided in \ref{apx-sec:proofs} first expresses $\theta_m^\downarrow$ in a variational form as
\begin{align*}
	\theta_m^\downarrow & = \min_{S\in \binom{[d]}{d-m+1}} \max_{u\in \triangle^{d-m+1}} \inp{u,\hat{\theta}_S} - \inp{u,X_S}\,.
\end{align*}
Then it establishes \eqref{eq:CI-lower} by leveraging a uniform instance-dependent tail bound for $\inp{u,X_S}$ and taking the union bound over $S\in \binom{[d]}{d-m+1}$. We only need to recover \eqref{eq:CI-instance-dependent}, the instance-dependent bound for $\inp{u,X_S}$, using \ref{thm:main-exponential}. Therefore, the core of the argument is basically the same argument we used in \ref{sssec:Gaussian-marginals-Thm1} with some modifications.

Recalling that $\triangle^\ell$ denotes the unit simplex in $\mbb R^\ell$, for any fixed $S\in \binom{[d]}{d-m+1}$ let
\begin{align*}
	\mc F = \left\{f_u  = \inp{u,\cdot}\st u\in \triangle^{d-m+1}\right\}\,.
\end{align*}
Furthermore, for a sufficiently small nonnegative integer $k_0$, let $\triangle^{d-m+1}_{k_0}$ denote the orthogonal projection of $\triangle^{d-m+1}$ onto the range of $\Sigma_{S,k_0}$. Taking $\mc N _{\epsilon/2}$ to be an $\epsilon/2$-net of $\triangle^{d-m+1}_{k_0}$ with respect to the metric induced by $\norm{\cdot}_{\Sigma_S}$ let
\begin{align*}
	\tilde{u}_{\epsilon/2} & = \argmin_{u'\in \mc N_{\epsilon/2}} \norm{u-u'}_{\Sigma_S}\,,
\end{align*}
and
\begin{align*}
	\hat{u}_\epsilon & = \begin{cases}
		                     \left(1-\frac{\epsilon}{2\norm{\tilde{u}_{\epsilon/2}}_{\Sigma_S}}\right)_+\tilde{u}_{\epsilon/2} & \text{if}\ \norm{\tilde{u}_{\epsilon/2}}_{\Sigma_S} > \norm{u}_{\Sigma_S}\,, \\
		                     \tilde{u}_{\epsilon/2}                                                                            & \text{otherwise}\,.
	                     \end{cases}
\end{align*}
Then, we have
\begin{align*}
	\mc A = \left\{f_u - f_{\hat{u}_\epsilon} = \inp{u-\hat{u}_\epsilon,\cdot}\st u\in \triangle^{d-m+1} \right\}\,,
\end{align*}
for which
\begin{align*}
	\gamma(\mc A;r,\ubar{\ell},1) & \lesssim \gamma_2(\mc V_{k_0} + \mc V^\perp_{k_0},\rho_2;\ubar{\ell} )                                      \\
	                              & \le \gamma_2(\mc V_{k_0} ,\rho_2;\ubar{\ell}) + \gamma_2(\mc V^\perp_{k_0,\epsilon} ,\rho_2;\ubar{\ell})\,,
\end{align*}
where again $\rho_2(x,y) = \norm{x-y}_2$, and
\begin{align*}
	\mc V_{k_0} & = \left\{\left(\Sigma_S-\Sigma_{S,k_0}\right)^{1/2}u\st u\in \triangle^{d-m+1}\right\}\,,
\end{align*}
and
\begin{align*}
	\mc V^\perp_{k_0,\epsilon} & = \left\{\Sigma_{S,k_0}^{1/2}(u-\hat{u}_\epsilon)\st u\in \triangle^{d-m+1}\right\}\,.
\end{align*}
We again can invoke the majorizing measures theorem \cite[Theorem 2.4.1]{Tal14} as well as the bound on the entrywise maximum of a Gaussian random vector \cite[Proposition 2.4.16]{Tal14}; with $Z\sim\mr{Normal}(0,I)$ we obtain
\begin{align*}
	\gamma_2(\mc V^\perp_{k_0,\epsilon} ,\rho_2;\ubar{\ell}) & \lesssim \E \sup_{v\in \mc V_{k_0}}\inp{v,Z} \\
	                                                         & \lesssim \sigma^*(S,k_0)\,,
\end{align*}
\begin{align*}
	\gamma_2(\mc V_{k_0} ,\rho_2;\ubar{\ell}) & \lesssim \E \sup_{v\in V^\perp_{k_0,\epsilon}} \inp{v,Z} \\& \le \sqrt{k_0}\epsilon\,,
\end{align*}
and thereby
\begin{align*}
	\gamma(\mc A;r,\ubar{\ell},1) & \lesssim \sigma^*(S,k_0) + \sqrt{k_0}\epsilon\,.
\end{align*}
We also have
\begin{align*}
	\mr{rad}_{2r+1}(\mc A) & \lesssim \sqrt{r}\sup_{u\in \triangle^{d-m+1}}\norm{u-\hat{u}_\epsilon}_{\Sigma_S}       \\
	                       & \le \sqrt{r}\sup_{u\in \triangle^{d-m+1}}(\norm{u}_{\Sigma_S - \Sigma_{S,k_0}}+\epsilon) \\
	                       & = \sqrt{r}\left(\norm{\sigma(S,k_0)}_\infty + \epsilon\right)\,.
\end{align*}
Therefore, if $k\ge \log(|\mc N_{\epsilon/2}|)\ge |A[\mc F]|$, it follows from \ref{thm:main-exponential} that with probability at least $1-2e^{-r}$, for all $u \in \triangle^{d-m+1}$ we have
\begin{align*}
	\inp{u,X_S} - \sqrt{2(r+k)}\norm{u}_{\Sigma_S} & \lesssim  \sigma^*(S,k_0) + \sqrt{r}\norm{\sigma(S,k_0)}_\infty + (\sqrt{r}+\sqrt{k_0})\epsilon\,.
\end{align*}
By the approximation $|\mc N_{\epsilon/2}| \le (1+4\sqrt{\norm{\Sigma_S}_\mr{op}}/\epsilon)^{k_0}$, it suffices to have $k\ge k_0 \log(1+ 4\sqrt{\norm{\Sigma_S}_\mr{op}}/\epsilon)$. In particular, using the fact that $\norm{\Sigma_S}_{\mr{op}} \le \tr(\Sigma_S) \le (d-m+1)\norm{\sigma}$ we can choose $\epsilon = \min\{\sigma^*/\sqrt{k_0}\,,\norm{\sigma}_\infty\}$ and $k\ge k_0 \log\left(1+ 4\sqrt{\norm{\Sigma_S}_\mr{op}}\max\{\sqrt{k_0}/\sigma^*,1/\norm{\sigma}_\infty\}\right)$. Therefore, assuming that
\begin{align*}
	\sqrt{\norm{\Sigma_S}_\mr{op}}\max\{\sqrt{k_0}/\sigma^*,1/\norm{\sigma}_\infty\} & \lesssim d\,,
\end{align*}
we conclude that for $k\gtrsim k_0 \log(d)$, with probability at least $1-2e^{-r}$, for all $u\in \triangle^{d-m+1}$ we have
\begin{align*}
	\inp{u,X_S} - \sqrt{2(r+k)}\norm{u}_{\Sigma_S} & \lesssim  \sigma^*(S,k_0) + \sqrt{r}\norm{\sigma(S,k_0)}_\infty\,.
\end{align*}
By union bound, with probability at least $1-2e^{-r}$, for all $S\in \binom{[d]}{d-m+1}$ and $u\in \triangle^{|S|}$ we have
\begin{align*}
	\inp{u,X_S} - \sqrt{2(r+ m + m\log(d/m)+k)}\norm{u}_{\Sigma_S} & \lesssim  \sigma^*(S,k_0) + \sqrt{r+ m + m\log(d/m)}\norm{\sigma(S,k_0)}_\infty\,.
\end{align*}
Using this inequality in variational expression for $\theta^\downarrow_m$ recovers \eqref{eq:CI-lower} up to the constant factors. The derivations for the upper bound \eqref{eq:CI-upper} can be carried out similarly by modifying the corresponding parts of the proof of \ref{prop:CI-Gaussian-top-m}.

\subsubsection{An abstraction of the example}
The instance-dependent bound was useful in this example thanks to the variational characterization of $\theta^\downarrow_m$. More generally, we can consider estimating $Y(\theta)$ given the noisy observation $\hat{\theta} = \theta + X$, where the function $Y\st \mbb R^d \to \mbb R$ is the minimum over $S\in \mc S$ of convex (lower-semicontinuous) functions $y_S(\cdot)$, i.e.,
\begin{align*}
	Y(x) & = \inf_{S\in \mc S} y_S(x)\,.
\end{align*}
Expressing $y_S(\cdot)$ using its convex conjugate $y_S^*(\cdot)$, we have an equivalent definition
\begin{align*}
	Y(x) & = \inf_{S\in \mc S} \sup_{u} \inp{u,x} - y^*_S(u)\,.
\end{align*}
Therefore,
\begin{align*}
	Y(\theta) & = \inf_{S\in \mc S} \sup_{u} \inp{u,\hat{\theta}} - \inp{u,X} - y^*_S(u)\,,
\end{align*}
which again is a variational formulation where the linear term $\inp{u,X}$ is exposed and can be approximated using instance-dependent tail bounds.

For example, if $\Theta$ is a $d_1\times d_2$ real matrix with $d_2\ge d_1$, the $m$-th largest singular value of $\Theta$ for $m\le d_1$, denoted by $\sigma_m(\Theta)$, can be expressed as
\begin{align*}
	\sigma_m(\Theta) & = \inf_{S\subseteq \mbb R^{d_1}\st \dim(S)=d_1-m+1} \sup_{U \in \mbb R^{d_1\times d_2}\st \norm{U}_*\le 1, \mr{range} (U)= S} \inp{U, \Theta}\,,
\end{align*}
where the infimum is taken over $d_1-m+1$-dimensional subspaces of $\mbb R^{d_1}$, and $\norm{U}_*$ and $\mr{range}(U)$, respectively, denote the nuclear norm and the range (or column space) of the matrix $U$.

\section{Tail Bounds Without the Exponential Moments}\label{sec:main-only-L1}
The results of \ref{sec:main-exponential} rely on the assumption that $\mc F$, the function class of interest, is a subset of (zero-mean) functions whose exponential moment is finite in a neighborhood of the origin. We may relax this assumption significantly by considering $\mbb V$ to be the vector space of zero-mean functions in $L_1(X)$. Then, using a variational approximation of quantile functions \cite[Theorem 2.3]{Pin14} for $\E_n g(X)$, we can define the analog of $T_r(\cdot)$ as
\begin{align}
	T^\sharp_{r,n}(g) & \defeq \inf_{t\in \mbb R} t + e^r \E \left((\E_n g(X) - t)_+\right)\,, \label{eq:variational-quantile-general}
\end{align}
where $(x)_+=\max(x,0)$ denotes the positive part of $x\in \mbb R$. Similarly, we can define
\begin{align*}
	\bar{T}^\sharp_{r,n}(g) & = \max\left\{T^\sharp_{r,n}(g),T^\sharp_{r,n}(-g)\right\}\,,
\end{align*}
which is a seminorm since it inherits convexity and subadditivity from the corresponding quantile approximation \cite[Theorem 2.3]{Pin14}. Equipped with the seminorm $\bar{T}^\sharp_{r,n}(\cdot)$, we can define the analogs of \eqref{eq:envelope-metric} and \eqref{eq:gamma} respectively as
\begin{align*}
	\varrho^\sharp_{r,n}(g,\mc H) & = \inf_{h\in \mc H} \bar{T}^\sharp_{r,n}(g-h)\,,
\end{align*}
for any $\mc H\subseteq \mbb V$, and
\begin{align*}
	\gamma^\sharp(\mc A;\, r, \ubar{\ell},n) & = \inf_{\left(\mc A_i\right)_{i \ge 0}} \sup_{a\in \mc A}\ \sum_{\ell \ge \ubar{\ell}}\varrho^\sharp_{(r + (r+1)2^{\ell-\ubar{\ell}}),n}(a, \mc A_\ell)\,.
\end{align*}
where, as in \eqref{eq:gamma}, $\mc A\subseteq \mbb V$, $\ubar{\ell}_\ge 0$, and the infimum is taken over an increasing admissible sequence $(\mc A_i)_{i\ge 0}$ of the subsets of $\mc A$. The corresponding radius of $\mc A$ is also denoted by
\begin{align*}
	\rad^\sharp_{r,n}(\mc A) & = \max_{a\in \mc A} \bar{T}^\sharp_{r,n}(a)\,.
\end{align*}
Therefore, we can refine \ref{thm:main-exponential} to the following theorem. We omit the proof as it is effectively the same as the proof of \ref{thm:main-exponential} with $\bar{T}_r(f)$ replaced by $\bar{T}_{r,n}^\sharp(f)$ for every $r\ge 0$ and $f\in \mbb V$ that appear in the proof.
\begin{thm}\label{thm:main-general} Let $A\st \mc F \to \mc F$ be a mapping such that
	\begin{align*}
		\bar{T}^\sharp_{r+k,n}(A[f]) & \le \bar{T}^\sharp_{r+k,n}(f)\,,\quad \text{for all}\ f\in \mc F\,,
	\end{align*}
	and
	\begin{align*}
		|A[\mc F]| & \le e^k\,,
	\end{align*}
	for some nonnegative integer $k$, where $A[\mc F] = \left\lbrace A[f]\st f\in \mc F\right\rbrace$ denotes the range of $A[\cdot]$. Furthermore, let
	\begin{align*}
		\mc A & =\left\lbrace f - A[f]\st f\in \mc F\right\rbrace\,.
	\end{align*}
	Setting $\ubar{\ell}=\lfloor\log_2(r/3)\rfloor$ for $r\ge \log (2)$,
	with probability at least $1-2e^{-r}$, for all $f\in \mc F$ we have
	\begin{align*}
		\E_n f(X) - \bar{T}^\sharp_{r+k,n}(f) - \min\left\{2\bar{T}^\sharp_{2r+1,n}(f-A[f]),\, \rad^\sharp_{2r+1,n}(\mc A)\right\} & \le  2\gamma^\sharp(\mc A;\, r, \ubar{\ell},n)\,.
	\end{align*}
	The bound can further be optimized with respect to the mapping $A[\cdot]$, which both $k$ and $\mc A$ depend on.
\end{thm}

While \ref{thm:main-general} applies with minimal requirements thanks to the generality of the definition \eqref{eq:variational-quantile-general}, it does not make the dependence on the sample size (i.e., $n$) transparent. To address this problem, the function class needs to be further restricted, allowing for an approximation of $T^\sharp_{r,n}(g)$ that reveals the role of $n$. Results of this type already established in the literature, e.g., in \citep{LT15} and \citep{Men16}, and in a specialized form in \citep{MP12}, by introducing a more refined ``scale-sensitive'' version of Talagrand's $\gamma$ functional, merely assuming that the functions of interest have finite moments of any order. We can reproduce similar bounds from \ref{thm:main-general} using the following lemma.

\begin{lem}\label{lem:quantile-to-moment}
	Let $g\in \mbb V$ be a zero-mean function with finite $p$-th moment for some $p\ge 2$. Then, for $t>0$ we have
	\begin{align*}
		\E\left(\left(\E_n g(X) - t\right)_+\right) & \le \left(2\sqrt\frac{p}{n}\norm{g}_{\psi_{2,p}}\right)^p\frac{t^{-p+1}}{p-1}\,,
	\end{align*}
	where\footnote{The defined norm is denoted by $\norm{\cdot}_{(p)}$ in \citep{Men16}. Viewing this norm as an ``incomplete'' sub-Gaussian norm, we use the more indicative notation $\norm{\cdot}_{\psi_{2,p}}$ instead.}
	\begin{align}
		\norm{g}_{\psi_{2,p}} & \defeq \sup_{q\in [1,p]}\frac{\norm{g}_{L_q}}{\sqrt{q}}\,.\label{eq:partial-subGaussian-norm}
	\end{align}
\end{lem}
\begin{proof}
	We can apply Markov's inequality and Gin\`{e}-Zinn symmetrization (see, e.g., \cite[Lemma 6.4.2]{Ver18}) to obtain
	\begin{align*}
		\E\left((\E_n g(X) - t)_+\right) & = \int_{t}^\infty \P\left(\E_n g(X) \ge y\right)\d y                                                      \\
		                                 & \le \int_{t}^\infty \frac{\norm{\E_n g(X)}_{L_p}^p}{y^p}\d y                                              \\
		                                 & \le \left(\frac{2}{n}\right)^p \norm{\sum_{i=1}^n \varepsilon_i g(X_i)}_{L_p}^p\: \frac{t^{-p+1}}{p-1}\,.
	\end{align*}
	where $(\varepsilon_i)_{i\ge 1}$ is a sequence i.i.d. Rademacher random variables (independent of the $X_i$s). Furthermore, the moments of $\sum_{i=1}^n \varepsilon_i g(X_i)$, as a sum of i.i.d. symmetric random variables, can be bounded using a result due to \citet[Corollary 2]{Lat97} which yields
	\begin{align*}
		\norm{\sum_{i=1}^n \varepsilon_i g(X_i)}_{L_p} & \le \sup \left\{\frac{p}{q}\left(\frac{n}{p}\right)^{1/q}\norm{g}_{L_q}\st \max(2,p/n)\le q \le p\right\} \\
		                                               & \le \sqrt{pn}\sup_{q\in [1,p]}\frac{\norm{g}_{L_q}}{\sqrt{q}}\,.
	\end{align*}
	Recalling the definition of $\norm{g}_{\psi_{2,p}}$ in \eqref{eq:partial-subGaussian-norm}, the result follows by combining the above inequalities.
\end{proof}

Using \ref{lem:quantile-to-moment} we can bound $T^\sharp_{r,n}(g)$ in terms of $\norm{g}_{\psi_{2,p}}$. In particular, evaluating the argument of the infimum on the right-hand side of \eqref{eq:variational-quantile-general} at $t = 2\sqrt{p/n}e^{r/p}\norm{g}_{\psi_{2,p}}$ reveals that
\begin{align*}
	T^\sharp_{r,n}(g) & \le \frac{2p}{p-1}\sqrt{\frac{p}{n}}e^{r/p}\norm{g}_{\psi_{2,p}}\,.
\end{align*}
If $g(X)$ has a finite moment of order $p=r\ge 2$, then the above inequality reduces to
\begin{align*}
	T^\sharp_{r,n}(g) & \le 4e \norm{g}_{\psi_{2,r}}\sqrt{r/n}\,.
\end{align*}
Therefore, if we further assume that the functions of interest have finite moments of arbitrary order, then
\begin{align*}
	\gamma^\sharp(\mc A; r,\ubar{\ell},n) & \lesssim \inf_{\left(\mc A_i\right)_{i \ge 0}} \sup_{a\in \mc A}\ \sum_{\ell \ge \ubar{\ell}}\sqrt{\frac{r + (r+1)2^{\ell-\ubar{\ell}}}{n}}\norm{a- \mc A_\ell}_{\psi_{2,r + (r+1)2^{\ell-\ubar{\ell}}}}\,,
\end{align*}
where we use the shorthand $\norm{a-\mc A_\ell}_{\psi_{2,r}}$ to denote the distance between $a\in \mc A$ and the set $\mc A_\ell$ with respect to $\norm{\cdot}_{\psi_{2,r}}$. Choosing $\ubar{\ell}$ as prescribed by \ref{thm:main-general}, we have $r + (r+1)2^{\ell-\ubar{\ell}} < 2^{\ell+4}$, thus
\begin{align*}
	\gamma^\sharp(\mc A; r,\ubar{\ell},n) & \lesssim \frac{\gamma^\flat (\mc A; \ubar{\ell}) }{\sqrt{n}}\,,
\end{align*}
where
\begin{align*}
	\gamma^\flat (\mc A; \ubar{\ell}) & \defeq \inf_{\left(\mc A_i\right)_{i \ge 0}} \sup_{a\in \mc A}\ \sum_{\ell \ge \ubar{\ell}} 2^{\ell/2}\norm{a- \mc A_\ell}_{\psi_{2,2^{\ell+4}}}\,,
\end{align*}
and we have the following corollary.

\begin{cor}
	Let $A\st \mc F \to \mc F$ be a mapping such that
	\begin{align*}
		\norm{A[f]}_{\psi_{2,r+k}} & \le \norm{f}_{\psi_{2,r+k}}\,,\quad \text{for all}\ f\in \mc F\,,
	\end{align*}
	and
	\begin{align*}
		|A[\mc F]| & \le e^k\,,
	\end{align*}
	for some nonnegative integer $k$, where $A[\mc F] = \left\lbrace A[f]\st f\in \mc F\right\rbrace$ denotes the range of $A[\cdot]$. Furthermore, let
	\begin{align*}
		\mc A & =\left\lbrace f - A[f]\st f\in \mc F\right\rbrace\,.
	\end{align*}
	Setting $\ubar{\ell}=\lfloor\log_2(r/3)\rfloor$ for $r\ge 2$,
	with probability at least $1-2e^{-r}$, for all $f\in \mc F$ we have
	\begin{align*}
		\E_n f(X) - 4e\sqrt{\frac{r+k}{n}}\norm{f}_{\psi_{2,r+k}} - 4e\sqrt{\frac{2r+1}{n}} \norm{f-A[f]}_{\psi_{2,2r+1}} & \lesssim \frac{\gamma^\flat(\mc A;\ubar{\ell},n)}{\sqrt{n}} \,.
	\end{align*}
\end{cor}

%% file: appendix-arxiv.tex
\begin{appendix}
	\section{Remaining Lemmas and Proofs}\label{apx-sec:proofs}

	\subsection*{Proofs of \ref{sec:introduction,sec:problem}}
	\begin{proof}[Proof of \ref{lem:T-r(g)}]
		Part \labelcref{itm:positive_homogenous} of the lemma follows from a straightforward change of variable.

		For part \labelcref{itm:zero_is_root}, we have \begin{align*}
			T_0(g) & = \inf_{\lambda \ge 0}\frac{\log \E e^{\lambda g(X)}}{\lambda}                                 \\
			       & \le \lim _{\lambda \downarrow 0} \frac{\log \E e^{\lambda g(X)}}{\lambda}                      \\
			       & = \lim _{\lambda \downarrow 0} \frac{\E\left(g(X)e^{\lambda g(X)}\right)}{\E e^{\lambda g(X)}} \\
			       & = 0\,,
		\end{align*}
		where the third and fourth line respectively follow from the l'H\^{o}pital's rule and the assumption that $g(X)$ is zero-mean. However, by Jensen's inequality we have $\E e^{\lambda g(X)}\ge e^{\lambda \E g(X)}=1$, which means that $T_0(g) \ge 0$. Therefore, we must have
		\begin{align*}
			T_0(g) & = 0\,.
		\end{align*}

		For part \labelcref{itm:concave_subadditive} observe that  $T_r(g)$ can be equivalently expressed as
		\begin{align*}
			T_r(g) & = \inf_{\theta \ge 0} \left(\theta r + \theta \log \E e^{g(X)/\theta}\right)       \\
			       & = - \sup_{\theta \ge 0} \left(-\theta r - \theta \log \E e^{g(X)/\theta}\right)\,.
		\end{align*}
		Since the supremum is the convex conjugate of $\theta \mapsto \theta \log \E e^{g(X)/\theta}$ evaluated at $-r$, we conclude that $r\mapsto T_r(g)$ is concave. The proved concavity together with part \labelcref{itm:zero_is_root} of the lemma, guarantee that for all $r,s\ge 0$ we have
		\begin{align*}
			\frac{r}{r+s}T_{r+s}(g) & = \frac{r}{r+s}T_{r+s}(g)+\frac{s}{r+s}T_0(g) \\
			                        & \le  T_{r}(g)\,,
		\end{align*}
		and
		\begin{align*}
			\frac{s}{r+s}T_{r+s}(g) & = \frac{s}{r+s}T_{r+s}(g)+\frac{r}{r+s}T_0(g) \\
			                        & \le  T_{s}(g)\,,
		\end{align*}
		which add up to
		\begin{align*}
			T_{r+s}(g) & \le T_{r}(g) + T_{s}(g)\,,
		\end{align*}
		proving the subadditivity of $r\mapsto T_r(g)$.

		To prove part \labelcref{itm:envelope} we readily have $\bar{T}_r(0) = 0$, and
		\begin{align*}
			\bar{T}_r(\alpha g) & = \max\{T_r(\alpha g),T_r(-\alpha g)\} \\
			                    & = |\alpha| \bar{T}_r(g)\,,
		\end{align*}
		for every $f\in \mbb V$ and nonzero real number $\alpha$. Therefore, it suffices to show that $\bar{T}_r(g)$ is convex in $f\in \mbb V$. We show that $T_r(\cdot)$ is convex, which implies the convexity of $\bar{T}_r(\cdot)$. Let us define $\kappa(g) = \log \E e^{g(X)}$ for $f\in \mbb V$, and denote by $\mbb V^*$ the dual space of $\mbb V$, i.e., the space of linear functionals on $\mbb V$ that are bounded in the sup norm. It follows from the H\"{o}lder's inequality that $\kappa(\cdot)$ is convex. We also define the convex conjugate of $\kappa(\cdot)$ as \[\kappa^*(w) = \sup_{f\in \mbb V}\inp{w,f} - \kappa(g)\,,\]
		for every $w\in \mbb V^*$. It can be shown that $\kappa(\cdot)$ is also lower semi-continuous which guarantees $\kappa(g)=\sup_{w\in \mbb V^*}\inp{w,g}-\kappa^*(w)$ for all $g\in \mbb V$. Our goal is to show that
		\begin{align}
			T_r(g) & = \sup_{w\in \mbb V^*\st \kappa^*(w)\le r}\inp{w,g}\,, \label{eq:T-r conjugate}		\end{align}
		which clearly proves the convexity of $T_r(g)$. For $r=0$ the identity \eqref{eq:T-r conjugate} holds trivially as $T_0(g)=0$ for all $g\in \mbb V$. Then, without loss of generality we may assume that $r>0$ and write the right-hand side of \eqref{eq:T-r conjugate} as
		\begin{align*}
			\sup_{w\in \mbb V^*\st \kappa^*(w)\le r}\inp{w,g} & = \sup_{w\in \mbb V^*} \inf_{\gamma \ge 0} \inp{w,g} - \gamma(\kappa^*(w) - r)\,.
		\end{align*}
		Straightforward calculations show that $\kappa^*(0) = 0 < r$. Therefore, the \emph{Slater's condition} is satisfied, and by invoking strong duality we can write
		\begin{align*}
			\sup_{w\in \mbb V^*\st \kappa^*(w)\le r}\inp{w,g} & = \sup_{w\in \mbb V^*} \inf_{\gamma \ge 0} \inp{w,g} - \gamma(\kappa^*(w) - r)  \\
			                                                  & = \inf_{\gamma \ge 0} \sup_{w\in \mbb V^*}  \inp{w,g} - \gamma(\kappa^*(w) - r) \\
			                                                  & = \inf_{\gamma\ge 0} \gamma \kappa(\gamma^{-1} g) + \gamma r                    \\
			                                                  & = T_r(g)\,,
		\end{align*}
		where the last equation follows by the change of variable $\lambda = 1/\gamma$.
	\end{proof}

	\begin{proof}[Proof of \ref{lem:Chernoff}]
		By the standard Chernoff bound, for any $T>T_{r/n}(f)$ we have
		\begin{align*}
			\P\left(\E_n f(X) > T\right) & \le \inf_{\lambda\ge 0}e^{n\log\E e^{\lambda f(X)}-n\lambda T}\,.
		\end{align*}
		It follows from the definition of $T_{r/n}(\cdot)$ that there exists $\lambda' \ge 0$ such that
		\[
			T_{r/n}(f) \le \frac{r/n + \log\E e^{\lambda' f(X)}}{\lambda'} < T\,.
		\]
		Therefore, we deduce
		\begin{align*}
			\P\left(\E_n f(X) > T\right) & \le e^{\log\E e^{\lambda' f(X)}-\lambda' T} \\
			                             & \le e^{-r}\,.
		\end{align*}
		and consequently
		\begin{align*}
			\P\left(\E_n f(X)\le T_r(f)\right) & = \lim_{T\downarrow T_{r/n}(f)} \P\left(\E_n f(X)\le T\right) \\
			                                   & \ge 1-e^{-r}\,.
		\end{align*}
	\end{proof}

	\subsection*{Proofs of \ref{sec:Examples}}
	\begin{proof}[Proof of \ref{prop:Gaussian-instance-dependent-tail}]
		Let $G\sim \mr{Normal(0,I)}$ be a standard normal random vector. Clearly, $\E_n \inp{u,X}= \inp{\frac{1}{\sqrt{n}}\Sigma^{1/2}u,G}$
		in distribution, with $\Sigma^{1/2}$ denoting the symmetric square root of the covariance matrix $\Sigma$. Let $\Sigma_k$ denote the best rank-$k$ approximation of $\Sigma$ with respect to the operator norm, and let $\pi_k G$ denote the orthogonal projection of $G$ onto the range of $\Sigma_k$. We have
		\begin{align*}
			 & \inp{\frac{1}{\sqrt{n}}\Sigma^{1/2} u,G} - \frac{\sqrt{2r}+\sqrt{k}}{\sqrt{n}}(u^\T\Sigma u)^{1/2}                                                                    \\
			 & \le    \inp{\frac{1}{\sqrt{n}}\Sigma^{1/2} u,G- \pi_k G} +\\
			 & \qquad \inp{\frac{1}{\sqrt{n}}\Sigma^{1/2} u,\pi_k G} - \frac{\sqrt{2r}+\sqrt{k}}{\sqrt{n}}(u^\T\Sigma_k u)^{1/2} \\
			 & \le  \frac{1}{\sqrt{n}}\norm{\Sigma^{1/2}(G-\pi_k G)}_2 + \\
			 & \qquad \sqrt{\frac{u^\T \Sigma_k u}{n}}\left(\norm{\pi_k G}_2-\sqrt{2r}-\sqrt{k}\right)\,,
		\end{align*}
		where the second line follows from the fact that $u^\T\Sigma_k u\le u^\T\Sigma u$, and the third line follows from the Cauchy--Schwarz inequality applied to each of the inner products.
		Using the Gaussian concentration inequality, with probability at least $1-e^{-r}$ we have
		\begin{align*}
			\norm{\Sigma^{1/2}(G-\pi_k G)}_2 & \le    \sqrt{\tr(\Sigma -\Sigma_k)}+  \sqrt{\norm{\Sigma - \Sigma_k}_{\mr{op}}}\,\sqrt{2r}\,,
		\end{align*}
		and similarly, with probability at least $1-e^{-r}$,
		\begin{align*}
			\norm{\pi_k G}_2 & \le \sqrt{k} + \sqrt{2r}\,.
		\end{align*}
		The upper bound for $S_k$ in \eqref{eq:Gaussian-Prop-upperbound} follows by combining the three derived inequalities and using the identities $\tr(\Sigma - \Sigma_k) = \sum_{i=k+1}^d \lambda_i$ and $\norm{\Sigma - \Sigma_k}_{\mr{op}} = \lambda_{k+1}$.

		To prove the lower bound for $S_k$, first observe that if $6r+3(\sqrt{2r}+\sqrt{k})^2 > d$ then \eqref{eq:Gaussian-Prop-tightness} holds trivially as its right-hand side vanishes to zero. Therefore, without loss of generality we may assume that $6r+3(\sqrt{2r}+\sqrt{k})^2 \le d$. We can express $S_k$ by its dual representation as
		\begin{align*}
			S_k & =\sup_{u}\inf_{x}\inp{u,X-x}-\frac{\sqrt{2r}+\sqrt{k}}{\sqrt{n}}\left(u^\T\Sigma u\right)^{1/2}+\norm{x}_{2} \\
			    & =\inf_{x}\sup_{u}\inp{u,X-x}-\frac{\sqrt{2r}+\sqrt{k}}{\sqrt{n}}\left(u^\T\Sigma u\right)^{1/2}+\norm{x}_{2} \\
			    & =\inf_{x^\T\Sigma^{-1}x\le (\sqrt{2r}+\sqrt{k})^2/n}\norm{X-x}_{2}       \,,
		\end{align*}
		where we used the strong duality on the second line, which holds by the Slater's condition. Using the strong duality again to simplify $S_k^2$, we have
		\begin{align*}
			S_k^{2} & =\inf_{x}\sup_{\beta\ge0}\norm{X-x}_{2}^{2}+\beta\left(x^\T\Sigma^{-1}x-\frac{(\sqrt{2r}+\sqrt{k})^2}{n}\right)                                                                                        \\
			        & =\sup_{\beta\ge0}\inf_{x}\norm{X-x}_{2}^{2}+\beta\left(x^\T\Sigma^{-1}x-\frac{(\sqrt{2r}+\sqrt{k})^2}{n}\right)                                                                                        \\
			        & =\sup_{\beta\ge0}\norm{\left(I-\left(I+\beta\Sigma^{-1}\right)^{-1}\right)X}_{2}^{2}+\\
					& \qquad \beta\left(\norm{\left(I+\beta\Sigma^{-1}\right)^{-1}X}_{\Sigma^{-1}}^{2}-\frac{(\sqrt{2r}+\sqrt{k})^2}{n}\right) \\
			        & =\sup_{\beta\ge0}X^{\T}\left(I-\left(I+\beta\Sigma^{-1}\right)^{-1}\right)X-\frac{\beta (\sqrt{2r}+\sqrt{k})^2}{n}                                                                                     \\
			        & \overset{\mr{dist.}}{=}\sup_{\beta\ge0}\sum_{i=1}^{d}\frac{\beta\lambda_{i}}{n(\lambda_{i}+\beta)}g_{i}^{2}-\frac{\beta (\sqrt{2r}+\sqrt{k})^2}{n}\,,
		\end{align*}
		where $g_i$'s are i.i.d. standard Gaussian random variables. With $a_i = \beta \lambda_i/(n(\beta + \lambda_i))$ for $i\in [d]$, for any fixed $\beta \ge 0$, using the Chernoff bound and the formula for the moment-generating function of $g_i^2$, with probability at least $1-e^{-r}$, we have
		\begin{align*}
			\sum_{i=1}^d a_ig_i^2 & \ge \sup_{c\ge 0} \frac{-r + \sum_{i=1}^d \log(1+2ca_i)/2}{c}\,.
		\end{align*}
		Therefore, we can guarantee with the same probability that
		\begin{align*}
			S_k^2 & \ge \sup_{\beta\ge 0} \sup_{c\ge 0} \frac{-r + \sum_{i=1}^d \log\left(1+\frac{2c\beta \lambda_i}{n(\beta+\lambda_i)}\right)/2}{c} -\frac{\beta (\sqrt{2r}+\sqrt{k})^2}{n}\,.
		\end{align*}
		Recall that $k' = \lceil 6r+3(\sqrt{2r}+\sqrt{k})^2\rceil$. Choosing $\beta = \lambda_{k'}$, and $c= n/(2\beta)$ we have
		\begin{align*}
			S_k^2 & \ge \frac{-r +\sum_{i=1}^d \log(1+\frac{\lambda_i}{\lambda_{k'} + \lambda_i})/2}{n/(2\lambda_{k'})} - \frac{(\sqrt{2r}+\sqrt{k})^2}{n}\lambda_{k'} \\
			      & \ge \frac{\lambda_{k'}\sum_{i=1}^d \log(1+\frac{\lambda_i}{\lambda_{k'} + \lambda_i})}{n} - \frac{2r+(\sqrt{2r}+\sqrt{k})^2}{n}\lambda_{k'}        \\
			      & \ge \frac{\sum_{i=1}^d \lambda_{k'}\lambda_i/(\lambda_{k'} + 2\lambda_i)}{n} - \frac{2r+(\sqrt{2r}+\sqrt{k})^2}{n}\lambda_{k'}\,,
		\end{align*}
		where we used the inequality $\log(1+x) \ge x/(x+1)$ for $x\ge 0$ on the last line. Splitting the sum into a sum over $i\le k'$, and a sum over $i>k'$, we have
		\begin{align*}
			\sum_{i\le k'} \lambda_{k'}\lambda_i/(\lambda_{k'} + 2\lambda_i) & \ge \frac{k'}{3}\lambda_{k'}\,,
		\end{align*}
		and
		\begin{align*}
			\sum_{i> k'} \lambda_{k'}\lambda_i/(\lambda_{k'} + 2\lambda_i) & \ge \frac{1}{3}\sum_{i>k'}\lambda_i\,.
		\end{align*}
		Therefore, we have
		\begin{align*}
			S^2_k & \ge \frac{1}{3n}\sum_{i=k'+1}^d \lambda_i\,.
		\end{align*}
	\end{proof}

	\begin{proof}[Proof of \ref{prop:CI-Gaussian-top-m}]
		We express $\theta^\downarrow_m$ in an equivalent min-max variational form as
		\begin{align}
			\theta^\downarrow_m & = \min_{S\in \binom{[d]}{d-m+1}} \max_{u\in \triangle^{d-m+1}} \inp{u, \theta_S}\nonumber                                               \\
			                    & = \min_{S\in \binom{[d]}{d-m+1}} \max_{u\in \triangle^{d-m+1}} \inp{u,\hat{\theta}_S} - \inp{u,X_S}\,.\label{eq:theta-variational-form}
		\end{align}
		For the prescribed nonnegative integer $k\le r$, let $\pi_S$ and $\pi_S^\perp$, respectively, denote the orthogonal projections onto the range and the nullspace of $\Sigma_{S,k}$. Then, with $G\sim \mr{Normal}(0,I)$ we can write
		\begin{align*}
			\inp{u, \pi_S X_S} & \overset{\mr{dist.}}{=} \inp{\Sigma_S^{1/2}\pi_S u, \pi_S G} \\
			                   & \le \norm{\Sigma_S^{1/2}\pi_S u}_2\norm{\pi_S G}_2           \\
			                   & = \norm{u}_{\Sigma_{S,k}}\norm{\pi_S G}_2                    \\
			                   & \le \norm{u}_{\Sigma_S}\norm{\pi_S G}_2\,.
		\end{align*}
		By the Gaussian concentration inequality, with probability at least $1-e^{-r}$, we have
		\begin{align*}
			\norm{\pi_S G}_2 & \le \sqrt{2(r+k)}\,,
		\end{align*}
		thereby, on the same event, for all $u\in \mbb R^{d-m+1}$ we have
		\begin{align}
			\inp{u, \pi_S X_S} & \le \norm{u}_{\Sigma_S} \sqrt{2(r+k)}\label{eq:CI-top-projection}\,.
		\end{align}
		Furthermore, recalling the definition of  $\sigma = \sigma(S,k)$, with probability at least $1-e^{-r}$ we have
		\begin{align*}
			\inp{u, \pi^\perp_S X_S} & \le \norm{u}_1 \norm{\pi_S^\perp X_{S}}_\infty                                              \\
			                         & \le \norm{u}_1 (\E( \norm{\pi_S^\perp X_{S}}_\infty) + \sqrt{2r}\norm{\sigma(S,k)}_\infty ) \\
			                         & \le \norm{u}_1 (C \sigma^*(S,k) + \sqrt{2r}\norm{\sigma(S,k)}_\infty)\,,
		\end{align*}
		where the second line follows from the Gaussian concentration inequality, and the third line follows from \cite[Proposition 2.4.16 and the remarks after Theorem 2.4.18]{Tal14} for some absolute constant $C>0$. Adding the derived inequalities, with probability at least $1-2e^{-r}$, for all $u\in \mbb R^{d-m+1}$, we can guarantee
		\begin{align}
			\inp{u,X_S} & = \inp{u, \pi^\perp_S X_S}+\inp{u, \pi_S X_S} \nonumber                                                                                        \\
			            & \le \norm{u}_1 (C \sigma^*(S,k) + \sqrt{2r}\norm{\sigma(S,k)}_\infty) + \norm{u}_{\Sigma_{S}}\sqrt{2(r+k)}\,. \label{eq:CI-instance-dependent}
		\end{align}
		Applying this bound in \eqref{eq:theta-variational-form}, for any fixed $S\in \binom{[d]}{d-m+1}$, with probability at least $1-2e^{-r}$, we have
		\begin{align*}
			\max_{u\in\triangle^{d-m+1}}\inp{u,\theta_S} & \ge \max_{u\in \triangle^{d-m+1}} \inp{u,\hat{\theta}_S} - (C \sigma^*(S,k) + \sqrt{2r}\norm{\sigma(S,k)}_\infty) - \norm{u}_{\Sigma_{S}}\sqrt{2(r+k)}\,.
		\end{align*}
		To obtain a lower bound for $\theta^\downarrow_m$, we can choose $S$ to be the indices of the $d-m+1$ smallest entries of $\theta$. But to be truly agnostic to the choice of $\theta$, we need to invoke the union bound and minimize the lower bound with respect to $S\in\binom{[d]}{d-m+1}$, at the cost of increasing $r$ by $m + m\log(d/m) > \log \tbinom{d}{m-1}$. The resulting inequality is then
		\begin{align*}
			\theta^\downarrow_m & \ge \min_{S\in \binom{[d]}{d-m+1}}\max_{u\in \triangle^{d-m+1}} \Big(\inp{u,\hat{\theta}_S} - (C \sigma^*(S,k) + \sqrt{2(r+m + m\log(d/m))}\norm{\sigma(S,k)}_\infty) \\
			                    & \hphantom{\min_{S\in \binom{[d]}{d-m+1}}\max_{u\in \triangle^d}} - \norm{u}_{\Sigma_{S}}\sqrt{2(r+m + m\log(d/m)+k)}\Big)\,,
		\end{align*}
		which, by identifying the expressions of $\beta_{r,m,k}$ and $Q_{S,\beta_{r,m,k}}(\cdot)$, is equivalent to \eqref{eq:CI-lower}. To establish the upper bound \eqref{eq:CI-upper}, observe that $\theta^\downarrow_m = - (-\theta)^\downarrow_{d-m+1}$, which allows us to reuse the inequalities above to derive an upper bounds for $\theta^\downarrow_m$ through the lower bound for $(-\theta)^\downarrow_{d-m+1}$.
	\end{proof}
	It is worth mentioning that for $k=0$, the left-hand side of \eqref{eq:CI-top-projection} vanishes, thereby we can improve the inequality \eqref{eq:CI-instance-dependent} to
	\begin{align*}
		\inp{u,X_S} & \le \norm{u}_1 (C \sigma^*(S,0) + \sqrt{2r}\norm{\sigma(S,0)}_\infty)\,.
	\end{align*}
	Consequently, for $k=0$ the corresponding bounds are in fact
	\begin{align*}
		\theta^\downarrow_m & \ge \min_{S\in \binom{[d]}{d-m+1}}\left(\max_{i\in S}\hat{\theta}_i - C\sigma^*(S,0) - \sqrt{2(r+m + m\log(d/m))}\norm{\sigma(S,0)}_\infty\right)\,,
	\end{align*}
	and
	\begin{align*}
		\theta^\downarrow_m & \le \max_{S\in \binom{[d]}{m}}\left(\min_{i\in S}\hat{\theta}_i + C\sigma^*(S,0) + \sqrt{2(r+m + m\log(d/m))}\norm{\sigma(S,0)}_\infty\right)\,.
	\end{align*}
	\begin{proof}[Proof of \ref{prop:lower-bound}]
		Le Cam's two point method \cite[Theorem 31.1]{PW24} (see also \cite[Lemma 1]{Yu97}) guarantees that
		\begin{align*}
			\sup_{\theta \in \Theta} \E (g(\hat{\theta})-\theta^\downarrow_m)^2 & \ge \sup_{\theta,\eta\in \Theta}\frac{1}{4}(\theta^\downarrow_m-\eta^\downarrow_m)^2\left(1-D_{\mr{TV}}(\P_\theta,\P_{\eta})\right)\,,
		\end{align*}
		where $D_{\mr{TV}}(\cdot,\cdot)$ denotes the total variation distance, and $\P_\theta = \mr{Normal}(\theta,\Sigma)$ and $\P_{\eta} = \mr{Normal}(\eta,\Sigma)$. The ``simplified'' Bretagnolle--Huber inequality \cite[Equation 2.25]{Tsy08} (see also \citep{Can23} for a broader context) guarantees that
		\begin{align}
			D_{\mr{TV}}(\P_\theta,\P_{\eta}) & \le 1- \frac{1}{2}e^{-(\theta-\eta)^\T\Sigma^{-1}(\theta-\eta)/2}\,,\label{eq:sBH}
		\end{align}
		using which we obtain
		\begin{align*}
			\sup_{\theta \in \Theta} \E (g(\hat{\theta})-\theta^\downarrow_m)^2 & \ge \sup_{\theta,\eta\in \Theta}\frac{1}{8}(\theta^\downarrow_m-\eta^\downarrow_m)^2 e^{-(\theta-\eta)^\T\Sigma^{-1}(\theta-\eta)/2}                   \\
			                                                                    & = \sup_{\theta,\eta\in \Theta}\sup_{b\in [0,1]}\frac{1}{8}(\theta^\downarrow_m-\eta^\downarrow_m)^2 b e^{-b(\theta-\eta)^\T\Sigma^{-1}(\theta-\eta)/2} \\
			                                                                    & \ge \frac{1}{8e\max\{1,2\kappa^2\}}\sup_{\theta,\eta\in \Theta}(\theta^\downarrow_m-\eta^\downarrow_m)^2\,,
		\end{align*}
		where the second line follows from the fact that $\Theta$ is star-shaped, and the third line follows from the inequality $\max_{b\in[0,1]} be^{-bz} \ge e^{-1}/\max\{1,z\}$ for $z\ge 0$.
		We can derive \eqref{eq:expected-LB} from this lower bound using the fact that
		\begin{align*}
			\sup_{\theta,\eta\in \Theta}(\theta^\downarrow_m-\eta^\downarrow_m)^2 & = \left(\sup_{\theta\in\Theta}\theta^\downarrow_m - \inf_{\eta\in \Theta}\eta^\downarrow_m\right)^2          \\
			                                                                      & = \left(\sup_{\theta\in\Theta}\theta^\downarrow_m + \sup_{\eta\in \Theta}\eta^\downarrow_{d-m+1}\right)^2\,,
		\end{align*}
		where the latter equation follows from the symmetry of the set $\Theta$, and the fact that $-\eta^\downarrow_m = (-\eta)^\downarrow_{d-m+1}$.

		Furthermore, it follows from the definition of the total variation distance and \eqref{eq:sBH} that for any $c \ge 0$ we have
		\begin{align*}
			|\P(|g(\hat{\theta}) - \theta^\downarrow_m| > c) - \P(|g(\hat{\eta}) - \theta^\downarrow_m| > c)| & \le  1- \frac{1}{2}e^{-(\theta-\eta)^\T\Sigma^{-1}(\theta-\eta)/2}\,.
		\end{align*}
		In particular, for any $c \le |\theta^\downarrow_m - \eta^\downarrow_m|/2$, together with the inequality
		\begin{align*}
			\P(|g(\hat{\eta}) - \theta^\downarrow_m| > \frac{1}{2}|\theta^\downarrow_m - \eta^\downarrow_m|) & \ge \P(|g(\hat{\eta}) - \eta^\downarrow_m| < \frac{1}{2}|\theta^\downarrow_m - \eta^\downarrow_m|)\,,
		\end{align*}
		which follows from the triangle inequality, we obtain
		\begin{align*}
			\P(|g(\hat{\theta}) - \theta^\downarrow_m| \ge c) + \P(|g(\hat{\eta}) - \eta^\downarrow_m| \ge c) & \ge   \frac{1}{2}e^{-(\theta-\eta)^\T\Sigma^{-1}(\theta-\eta)/2}\,.
		\end{align*}
		Therefore, if there exists a pair $\theta,\eta\in \Theta$ such that $(\theta-\eta)^\T\Sigma^{-1}(\theta-\eta) \le 2$ and $|\theta^\downarrow_m - \eta^\downarrow_m|/2 \ge c$, then
		\begin{align*}
			\sup_{\theta\in \Theta}\P(|g(\hat{\theta}) - \theta^\downarrow_m| \ge c) & \ge \frac{1}{2e}\,.
		\end{align*}
		The desired result follows by setting $c = \sup_{\theta,\eta\in \Theta} |\theta^\downarrow_m-\eta^\downarrow_m|/(3\max\{1,\sqrt{2}\kappa\})$ which meets the required conditions.
	\end{proof}

	\begin{lem}\label{lem:Bernstein}
		Let $Y\in[-1,1]$ be a zero-mean random variable. Then, we have
		\begin{align*}
			\inf_{\lambda \ge 0} \frac{r+\log \E e^{\lambda Y}}{\lambda} & \le \frac{1}{3}r + \sqrt{2\E( Y^2) r}\,.
		\end{align*}
	\end{lem}
	\begin{proof}
		For all $\lambda \in [-3, 3]$, we have
		\begin{align*}
			\E e^{\lambda Y} & = 1 + \sum_{m=2}^\infty \frac{\E Y^m}{m!}\lambda^m                                          \\
			                 & \le 1 + \sum_{m=2}^\infty \frac{\E Y^2}{m!}|\lambda|^m                                      \\
			                 & \le 1 +   \frac{\E Y^2}{2}\sum_{m=2}^\infty \lambda^2\left(\frac{|\lambda|}{3}\right)^{m-2} \\
			                 & \le 1 + \frac{\E Y^2}{2}\cdot\frac{\lambda^2}{1-|\lambda|/3}  \,.
		\end{align*}
		Since $\log(1+u)\le u$ for all $u>-1$, it follows that
		\begin{align*}
			\inf_{\lambda \ge 0} \frac{r+\log \E e^{\lambda Y}}{\lambda} & \le \inf_{\lambda\in [0,3]}\frac{r+\E(Y^2)\lambda^2/(2-2\lambda/3)}{\lambda}\nonumber \\
			                                                             & \le \frac{1}{3}r + \sqrt{2\E(Y^2)r}\,,
		\end{align*}
		where the second line follows by evaluating the argument of the infimum at $\lambda = 3\sqrt{r}/(\sqrt{r} + 3\sqrt{\E(Y^2)/2})$.
	\end{proof}

	\section{Bounding $T_r(f)$ in Orlicz Spaces}\label{apx-sec:Orlicz}
	The purpose of this subsection is to approximate $T_r(f)$ for $f\in \mbb V$, in situations where $\mbb V$ is an \emph{Orlicz space} of \emph{exponential type}. Orlicz spaces are one of the important function spaces studied in functional analysis and probability theory. These function spaces can be described by their corresponding \emph{Orlicz norm}s. For a convex increasing function $\psi\st{} [0, \infty) \to [0, \infty)$ with $\psi(0)=0$ the $\psi$-Orlicz norm of a random variable $Y$ is defined as
	\begin{align*}
		\norm{Y}_\psi & \defeq \inf\left\{u > 0 \st \E \psi\left(\frac{|Y|}{u}\right)\le 1\right\}\,.
	\end{align*}
	Special cases are the usual $p$-norms for $p\ge 1$, the sub-Gaussian norm, and the sub-exponential norm, respectively, corresponding to $\psi(t) = t^p$, $\psi(t)=e^{t^2}-1$, and $\psi(t) = e^t-1$. Other interesting cases are the Bernstein{\textendash}Orlicz norm corresponding to
	\begin{align*}
		\psi(t) & = e^{{(\sqrt{1+2Lt}-1)}^2/L^2}-1\,,
	\end{align*}
	for some parameter $L>0$, introduced by \citet{vdGL12}, as well as the Bennett{\textendash}Orlicz norm corresponding to
	\begin{align*}
		\psi(t) & =e^{2\left((1+Lt)\log(1+Lt)-Lt\right)/L^2}-1\,,
	\end{align*}
	for some parameter $L>0$, introduced by \citet{Wel17}.

	To express the general bounds presented in \ref{thm:main-exponential} when the underlying metric of interest imposed on $\mc F$ is induced by an Orlicz $\psi$-norm, it suffices to bound $T_r(f)$ in terms of $\norm{f}_\psi\defeq \norm{f(X)}_\psi$. The following simple lemma can provide such bounds.

	\begin{lem}\label{lem:T-r-Orlicz}		
		For every $f\in \mbb V$ we have
		\begin{align}
			T_r(f) & \le\inf_{\lambda \ge 0} \frac{r+\log\left(1+\int_0^\infty 2\lambda \left(e^{\lambda t}-1\right)/\left(\psi(t)+1\right) \d t\right)}{\lambda}\,\norm{f}_{\psi}\,.\label{eq:Orlicz-coeffs}
		\end{align}
	\end{lem}
	\begin{proof}
		Without loss of generality we may assume $f\ne 0$. The inequality follows by bounding the moment generating function of the zero-mean random variable $Y = f(X)/\norm{f(X)}_\psi$, which has a unit $\psi$-Orlicz norm, as
		\begin{align*}
			\E e^{\lambda Y} & = \E\left( e^{\lambda Y} - \lambda Y\right)                                                                    \\
			                 & \le \E \left(e^{\lambda |Y|} -\lambda |Y| \right)                                                              \\
			                 & = 1+\int_0^\infty \lambda \P\left(|Y|>t\right)\left(e^{\lambda t}-1\right) \d t                                \\
			                 & \le 1+\int_0^\infty \lambda\left(\E \psi(|Y|)+1\right)\left(e^{\lambda t}-1\right)/\left(\psi(t)+1\right) \d t \\
			                 & = 1+\int_0^\infty 2\lambda \left(e^{\lambda t}-1\right)/\left(\psi(t)+1\right) \d t\,.\qedhere
		\end{align*}
	\end{proof}

	For exponential type Orlicz norms, defined below, we have the following proposition that provides a more explicit approximation for $T_r(f)$ in terms of $\norm{f}_\psi$.

	\begin{prop}\label{prop:exponential-type}
		Let $\norm{\cdot}_\psi$ be an Orlicz norm of exponential type, meaning that \begin{align*}
			\psi(t) & = e^{\phi(t)}-1\,
		\end{align*}
		for a convex and increasing function $\phi \st{} [0,\infty)\to [0,\infty)$ with $\phi(0)=0$. Furthermore, let $\phi^*(\cdot)$ denote the convex conjugate of $\phi(\cdot)$, i.e.,
		\begin{align*}
			\phi^*(\lambda) & = \sup_{t\ge 0} \left(\lambda t - \phi(t)\right)\,.
		\end{align*}
		If for some $M>0$ we have
		\begin{align}
			\inf_{\lambda \ge 0}\ \frac{e^{\phi^*(\lambda)}-1}{\lambda^2} & \ge M\int_0^\infty t e^{-\phi(t)/2}\d t\,,\label{eq:conversion-factor}
		\end{align}
		then for every $f\in \mbb V$ we have
		\begin{align}
			T_r(f) & \le\max\{3,3/\sqrt{2M}\}\phi^{-1}(2r/3)\,\norm{f}_\psi\,.\label{eq:Orlicz-coeffs-exp-type}
		\end{align}
	\end{prop}

	\begin{proof}
		Since $\phi^*(\cdot)$ is the convex conjugate of $\phi(\cdot)$, for every $\lambda, t\ge 0$ we can write
		\begin{align*}
			\lambda t \le \frac{1}{2}\phi(t) + \frac{1}{2}\phi^*(2\lambda)\,.
		\end{align*}
		Applying this bound to \eqref{eq:Orlicz-coeffs} of \ref{lem:T-r-Orlicz} we have
		\begin{align*}
			T_r(f) & \le\inf_{\lambda \ge 0} \frac{r+\log\left(1+\int_0^\infty 2\lambda \left(1-e^{-\lambda t}\right)e^{\lambda t - \phi(t)} \d t\right)}{\lambda}\,\norm{f}_\psi \\
			       & \le \inf_{\lambda \ge 0} \frac{r+\log\left(1+2\lambda^2 e^{\phi^*(2\lambda)/2}\int_0^\infty t e^{- \phi(t)/2} \d t\right)}{\lambda}\,\norm{f}_\psi\,,
		\end{align*}
		where we also used the inequality $1 - e^{-\lambda t}\le \lambda t$. It follows from \eqref{eq:conversion-factor} that
		\begin{align*}
			2\lambda^2 \int_0^\infty t e^{-\phi(t)/2}\d t & \le e^{\phi^*(\sqrt{2/M}\lambda)}-1\,.
		\end{align*}
		Then, using the fact that $\phi^*(\cdot)$ is nonnegative, we have
		\begin{align*}
			T_r(f) & \le\inf_{\lambda \ge 0} \frac{r+\phi^*(2\lambda)/2 + \phi^*(\sqrt{2/M}\lambda)}{\lambda}\,\norm{f}_\psi\,.
		\end{align*}
		Furthermore, because $\phi^*(\cdot)$ is increasing, we can write
		\begin{align*}
			\phi^*(2\lambda)/2+\phi^*(\sqrt{2/M}\lambda) & \le \frac{3}{2}\phi^*\left(\max\{2,\sqrt{2/M}\}\lambda\right)\,.
		\end{align*}
		Therefore, we conclude that
		\begin{align*}
			T_r(f) & \le\inf_{\lambda \ge 0} \frac{r+\frac{3}{2}\phi^*\left(\max\{2,\sqrt{2/M}\}\lambda\right)}{\lambda}\,\norm{f}_\psi \\
			       & = \max\{3,3/\sqrt{2M}\}\phi^{-1}(2r/3)\,\norm{f}_\psi\,.\qedhere
		\end{align*}

	\end{proof}

	It is worth mentioning that the constants appearing in the proposition are not necessarily optimal. In fact, the result may be improved for example by using the bound $1-e^{-\lambda t}\le \min\{\lambda t,1\}$ instead of the inequality $1-e^{-\lambda t}\le \lambda t$ that is used in the current proof. We did not pursue these refinements intending to obtain relatively simpler expressions.

	Let us quantify the result of \ref{prop:exponential-type} when $\norm{\cdot}_\psi$ is the sub-Gaussian Orlicz norm, and when it is the Bernstein--Orlicz norm. In the sub-Gaussian case, we have $\phi(t)=t^2$ and $\phi^*(\lambda) =  \bbone(\lambda \ge 0)\lambda^2/4$. It is easy to verify that \eqref{eq:conversion-factor} holds with $M = 1/4$. Therefore, for the sub-Gaussian Orlicz norm, \eqref{eq:Orlicz-coeffs-exp-type} reduces to
	\begin{align*}
		T_r(f) & \le \sqrt{12r}\,\norm{f}_{\psi}\,.
	\end{align*}

	In the case of Bernstein--Orlicz norm, $\phi(t) = (\sqrt{1+2Lt}-1)^2/L^2$. By the change of variable $t = \left((Lu+1)^2-1\right)/(2L)$ and using standard Gaussian integral formulas we can calculate the integral on the right-hand side of \eqref{eq:conversion-factor} as
	\begin{align*}
		\int_0^\infty t e^{-\left(\sqrt{1+2Lt}-1\right)^2/(2L^2)}\d t & = \sqrt{\frac{\pi}{8}}L + 1\,.
	\end{align*}
	Furthermore, with some straightforward calculations we can show that the convex conjugate of $\phi(\cdot)$ is
	\begin{align*}
		\phi^*(\lambda) & = \begin{cases}
			                    0\,,                                            & \lambda<0\,,           \\
			                    \frac{\lambda^2}{4\left(1-L\lambda/2\right)}\,, & \lambda \in [0,2/L)\,, \\
			                    \infty\,,                                       & \lambda > 2/L\,.
		                    \end{cases}
	\end{align*}
	Therefore, for $\lambda \ge 0$, we have
	\begin{align*}
		e^{\phi^*(\lambda)}-1 & \ge \phi^*(\lambda)  \ge \frac{\lambda^2}{4}\,.
	\end{align*}
	Consequently, \eqref{eq:conversion-factor} holds if
	\begin{align*}
		M & = \frac{1}{\sqrt{2\pi}L+4}\,,
	\end{align*}
	for which \eqref{eq:Orlicz-coeffs-exp-type} reduces to
	\begin{align*}
		T_r(f) & \le 3(\sqrt{\pi/2}L+2)^{1/2}\phi^{-1}(2r/3)\,\norm{f}_\psi             \\
		       & = (\sqrt{\pi/2}L+2)^{1/2}\left(Lr + \sqrt{6r}\right)\,\norm{f}_\psi\,.
	\end{align*}
\end{appendix}